\documentclass
[12pt]
{article}

\usepackage{amsmath}
\usepackage{amssymb}
\usepackage{amsthm}

\newcommand{\ga}{\gamma}
\newcommand{\e}{\varepsilon}
\newcommand{\la}{\lambda}

\newtheorem{theorem}{Theorem}[section]
\newtheorem{lemma}[theorem]{Lemma}
\newtheorem{remark}[theorem]{Remark}
\newtheorem{proposition}[theorem]{Proposition}

\newtheorem{corollary}[theorem]{Corollary}

\numberwithin{equation}{section}
\usepackage{color}
\usepackage[nottoc,notlot,notlof]{tocbibind}
\begin{document}
\title{
Infinite concentration and oscillation estimates for supercritical semilinear elliptic equations in discs. II
}
\author{Daisuke Naimen
}
\date{\small{Muroran Institute of Technology, 27-1, Mizumoto-cho, Muroran-shi, Hokkaido, 0508585, Japan}}
\maketitle
\begin{abstract} This paper is the latter part of our series concerning infinite concentration and oscillation phenomena on  supercritical semilinear elliptic equations in discs. Our supercritical setting admits two types of  nonlinearities, the Trudinger-Moser type $e^{u^p}$ with $p>2$ and the multiple exponential one $\exp{(\cdots(\exp{(u^m)}))}$ with $m>0$. In the first part, we accomplished the analysis of infinite concentration phenomena on any blow-up solutions. In this second part, we proceed to the study of infinite oscillation phenomena   based on concentration estimates obtained in the first part.  As a result, 
we provide a precise description of the asymptotic shapes of the graphs of blow-up solutions near the origin. Two types of oscillation and intersection properties are observed here depending on the choice of the growth. Moreover, it allows us to show several oscillation behaviors  around singular solutions with a suitable asymptotic behavior. This leads to  a natural sufficient condition for infinite oscillations of bifurcation diagrams which ensure the existence of infinitely many solutions. We successfully apply this condition to certain classes of nonlinearities including  the two types mentioned above. 
\end{abstract}
\tableofcontents
\section{Introduction}\label{sec:intr}
This paper is the second part of our series  on the infinite concentration and oscillation analysis of supercritical semilinear elliptic equations in discs. As in the first part \cite{N3}, we begin with considering the next problem,
\begin{equation}\label{p0}
\begin{cases}
-\Delta u=\la f(x,u), \ u>0&\text{ in }\Omega,\\
u=0&\text{ on }\partial \Omega,
\end{cases}
\end{equation}
where $\Omega$ is a bounded domain in $\mathbb{R}^2$ with smooth boundary  $\partial \Omega$, $\la>0$  a parameter, and $f:\bar{\Omega}\times [0,\infty)\to \mathbb{R}$ a nonnegative continuously differentiable function. In our main discussion, we impose the generalized supercritical exponential growth on $f$ mentioned in Subsection \ref{sub:ge} below. Our goal is to provide  infinite oscillation estimates for \eqref{p0} based on our infinite concentration estimates in \cite{N3}.  

The main introduction of our series is given in Section 1 in \cite{N3}. Hence, we here recall only the previous results  on oscillation phenomena on \eqref{p0}.  First, we note that if  $f(x,u)=e^u$, \eqref{p0} is known as the  Gelfand problem \cite{G}. In this case, Joseph-Lundgren \cite{JL} completes the classification of the  bifurcation diagram of the solutions $(\la,u)$  of \eqref{p0}   when $\Omega$ is an $N$-dimensional ball with $N\ge1$. From the result, one sees  that  if $1\le N\le 2$, then the diagram, which  emanates from the origin $(\la,u)=(0,0)$, turns exactly once and goes to the point $(\la,u)=(0,\infty)$ at infinity. On the other hand, one observes that if $3\le N\le 9$, then  it  turns infinitely many times around the axis $\la=\la^*$ for some suitable value $\la^*>0$. In particular, one ensures the existence of infinitely many solutions $u$ of \eqref{p0} for $\la=\la^*$ and many solutions for $\la$ around it. Finally, one finds that if $N\ge10$, then it has no turning point and goes to the point $(\la^*,\infty)$ at infinity for some suitable value $\la^*>0$. Note that in the second case above, we observe an infinite oscillation behavior which is one of our main subjects in this paper. 
 Recently, several authors accomplish suitable extensions of the result above to general supercritical problems in the case $N\ge3$. We refer the reader to  \cite{Mi2}, \cite{KW}, and \cite{MN} and references therein. See also the work on convex domains in \cite{FZ}. We remark that, in these works, the analysis of supercritical problems in dimension two  was left open. 

Concerning the supercritical problem of \eqref{p0} in dimension two, there are two earlier works \cite{AtPe} and \cite{McMc}  in discs for the Trudinger-Moser type growth $f(x,u)\sim e^{u^p}$  with $p>2$. An interesting supercritical phenomenon is suggested in \cite{McMc} as ``the  bouncing process" via some  heuristic and qualitative arguments.  We shall clarify this observation from the microscopic and more general  view points.  

Very recently, after decades from the work \cite{McMc}, new developments  are established  by \cite{N1}, \cite{N2}, and \cite{Ku} in discs. In \cite{N2}, based on the precise infinite concentration estimates in \cite{N1}, the author shows that blow-up solutions oscillate around  singular solutions infinitely many times for certain classes of the Trudinger-Moser type  nonlinearities. Then, following the idea in \cite{Mi} and with the aid of the singular solutions found in \cite{FIRT} and \cite{GGP}, he proves that the bifurcation diagram of \eqref{p0} oscillates around the axis $\la=\la^*$  infinitely many times for a suitable value $\la^*>0$.  On one hand, \cite{Ku} proves that the bifurcation diagram  has infinitely many turning points under his generalized setting by applying the analytic bifurcation theory.

Here, we remark that the author's work \cite{N2}  succeeds in proving not only the divergence of the number of turning points (obtained in \cite{Ku}) but also the oscillation around a suitable number $\la=\la^*$ which ensures the existence of infinitely many solutions $u$ for $\la=\la^*$. However,  his setting covers only the Trudinger-Moser type  growth   while the generalized settings in the related works \cite{Ku} and  \cite{FIRT} admit also the multiple exponential growth like $f(x,u)\sim \exp{(\cdots(\exp{(u^m)}))}$ with $m>0$. Hence one natural question is that whether we can accomplish an extension of the author's result in \cite{N2} to a generalized growth case including the multiple exponential one. In this paper, we shall give an answer based on  the results in the former part \cite{N3}. 

In  \cite{N3}, we succeed in establishing infinite concentration estimates for supercritical problems of \eqref{p0}  with generalized exponential nonlinearities which admit both of the Trudinger-Moser type and the multiple exponential ones.  Especially, we detect infinite sequences of concentrating parts on any blow-up solutions and accomplish the precise characterization of each concentration via  the limit equation with the suitable energy recurrence formulas. We remark that two types of infinite concentration behaviors are observed there. One is for  the Trudinger-Moser case and the other is for the multiple exponential case. The former one is an exact extension of the result in \cite{N1} while the latter shows a new behavior. See more precise discussions in the subsections below.

Now, our main aim of this paper is to  establish infinite oscillation estimates based on  the concentration estimates obtained in \cite{N3} under the generalized setting. Then we apply it to demonstrate several infinite oscillation behaviors of blow-up solutions around singular solutions. This observation contains new results on the infinite oscillation of the bifurcation diagram which yields the existence of infinitely many solutions of \eqref{p0}. 

More precisely, we first show that the infinite concentration estimate in \cite{N3} turns into a precise description of the asymptotic shapes of the graphs of blow-up solutions near the origin. We here observe two types of infinite oscillation behaviors reflecting the fact that there are two types of infinite concentration behaviors depending on the choice of the growth as mentioned above. Especially, the oscillation behavior in the multiple exponential case is new. Moreover, in both cases, it shows that the infinite sequence of concentrating parts can be observed  as an infinite sequence of bumps on the graphs. Using the precise estimates for the tops and bottoms of the bumps, we prove that any blow-up solutions oscillate  infinitely many times around singular solutions with a suitable asymptotic behavior. Then,  using  the idea of \cite{Mi}, we show that it induces  infinite boundary (or zero point) oscillations of blow-up solutions of the corresponding shooting type problems. This leads to  an interesting  global picture  of our concentration and oscillation phenomena studied in this series. That is, the infinite boundary oscillations generate an infinite number of bumps one after another from the  boundary and then, they climb up the graphs of blow-up solutions from the bottom to the top and finally, due to the supercritical growth, grow up as the infinite sequence of bubbles observed in \cite{N3}. Finally, as a main application of the discussion above, we provide a natural sufficient condition for several infinite oscillation behaviors of blow-up solutions around singular solutions. This result contains the assertion for the infinite oscillation of the bifurcation diagram which implies the existence of infinitely many solutions of \eqref{p0}. We also demonstrate that this condition is  applied for certain classes of nonlinearities,  considered in \cite{FIRT}, including the two types noted above.  This  gives a positive answer to the question noted above. 

A novelty of this paper is that we clarify  infinite oscillation phenomena on blow-up solutions from the view point of the infinite concentration phenomena under our general supercritical  setting. In particular,  we establish an extension of the result in \cite{N2} to wider and new classes of supercritical nonlinearities. This contains a new infinite oscillation behavior in the multiple exponential case. Among others,  to our best knowledge, this paper is the first work which arrives at the Joseph-Lundgren type multiplicity result, noted in the second paragraph above, for a class of multiple exponential nonlinearities in dimension two. Moreover, we believe that our sufficient condition mentioned above is useful to obtain the desired conclusion for wider classes of the two types of nonlinearities. This is  left for our future works. 

 Now, we begin with mentioning the setting and main results in the first part \cite{N3}. After that, we will state our main theorems base on them in the next section. 
\subsection{Radial setting with generalized exponential growth}\label{sub:ge} 
In this paper,  we study  \eqref{p0} with  $\Omega=D$ where $D$ is the unit disc centered at the origin. Moreover, we introduce the next assumptions on the nonlinearity $f$ as in \cite{N3}.  In the following,  we set  $\exp_1{(t)}=\exp{(t)}$ and $\exp_k{(t)}=\exp{(\exp_{k-1}{(t)})}$ for all $k\in \mathbb{N}$ with $k\ge2$ by induction  for all $t\ge0$.  
\begin{enumerate}
\item[(H0)]  $f(x,t)$ has the form $f(x,t)=h(|x|)f(t)$ with continuously differentiable functions $h:[0,1]\to (0,\infty)$ and $f:[0,\infty)\to [0,\infty)$. Moreover, there exists a value $t_0\ge0$ such that $f(t)>0$ for all $t\ge t_0$ and  $f\in C^2([t_0,\infty))$. Finally, we put $g(t)=\log{f(t)}$ for all $t\ge t_0$. 
\item[(H1)] We have the next (i) and (ii). 
\begin{enumerate}
\item[(i)] $g'(t)>0$ and  $g''(t)>0$ for all $t\ge t_0$ and there exists a pair $(q,p)\in \{1\}\times (0,\infty]\cup (1,\infty)\times (0,\infty)$ of values   such that  
\begin{equation}\label{g1}
\lim_{t\to \infty}\frac{g'(t)^2}{g(t)g''(t)}=q
\end{equation}
and 
\begin{equation}\label{g2}
\lim_{n\to \infty}\frac{t g'(t)}{g(t)}=p.
\end{equation}
\item[(ii)] In the case $q=1$, $tg'(t)/g(t)$ is nondecreasing for all $t\ge t_0$ and there exist a number $k\in \mathbb{N}$  and a function $\hat{g}\in C^2([t_0,\infty))$ such that $f(t)=\exp_k(\hat{g}(t))$ and $\hat{g}'(t)/\hat{g}(t)$ is nonincreasing 
for all $t\ge t_0$. 
\end{enumerate}
\item[(H2)] It holds that $\displaystyle\inf_{t>0}\frac{f(t)}{t}>0$. 
\end{enumerate} 
We always suppose (H0) without further comments throughout this paper. Moreover, we focus on radially symmetric classical solutions $u$ of \eqref{p0}.  (Note that  when $h$ is constant, the result in \cite{GNN} shows that any solution of \eqref{p0} is radially symmetric.) Then we may regard  $u=u(|x|)$ ($x\in \bar{D}$) and thus, \eqref{p0} turns into an ordinary differential equation for $u=u(r)$ $(r\in[0,1])$,
\begin{equation}\label{p}
\begin{cases}
- u''-\frac1r u'=\la h f(u), \ u>0\text{\ \  in }(0,1),\\
u'(0)=0=u(1).
\end{cases}
\end{equation}
We remark that we are always considering solutions in $C^2([0,1])$. Moreover, thanks to (H0), we readily see that each solution  $u$ of \eqref{p} with $u(0)>t_0$ is strictly decreasing (see Subsection 1.1 in \cite{N3}). In particular, $u(0)$ attains the maximum on $[0,1]$. 

Next let us mention  (H1) which is the essential condition for our infinite concentration and oscillation estimates. This is  based on the idea of the generalized H\"{o}lder conjugate exponent introduced by Dupaigne-Farina \cite{DF}. As we see by  Lemma \ref{lem:pq} and several examples  below, the pair $(q,p)$ of numbers satisfies  $1/p+1/q=1$ where $1/\infty=0$ and we can regard   $p$ as the growth rate of $g$ at infinity. Indeed, typical examples for the case $q>1$ are given by  functions $f$ with the Trudinger-Moser type growth such that
\[
f(t)=t^me^{t^p+c t^{\bar{p}}},\ \ \ e^{t^p (\log{t})^l}
\]
for all large $t>0$ with  any given numbers  $c,m,l\in \mathbb{R}$, $p>1$, and $\bar{p}>0$ such that  $1/p+1/q=1$ and $\bar{p}<p$. The former one can be considered also in the previous works \cite{N1} and \cite{N2} while the latter one is a new case. On the other hand, a typical example for the case $q=1$ is given by a function $f$ with the multiple exponential growth such like
\[
f(t)=\exp_k{(t^m(\log{t})^l)}
\]
for all large $t>0$ with any given values $k\in \mathbb{N}$ with $k\ge2$, $m>0$, and $l\in \mathbb{R}$. Clearly, this is also a new case which  can not be treated in \cite{N1} and \cite{N2}. 

As we see from the facts and examples above, we are interested in the case $1\le q<2$ which gives the supercritical growth $p>2$. Finally, we note that  (ii) of (H1) seems to be a technical assumption used only for the proof of Lemmas \ref{lem:g2} and \ref{lem:g4} below. In other words, we can accomplish all the proof by only (i) of (H1) with these lemmas. Hence, we may employ the conclusions of those lemmas as our (possibly weaker) assumptions instead of (ii).  

The last condition (H2) is technically assumed  when we want to ensure the upper bound for $\la$. See Lemma \ref{lem:kap} below. Next we proceed to the introduction of  our key energy recurrence formulas.  
\subsection{Energy recurrence formulas and limit profiles}\label{sub:bf}
In the first part \cite{N3}, we complete the precise quantification of the infinite concentration behavior by using  suitable sequences $(a_k)$, $(\delta_k)$, and $(\eta_k)$ of positive numbers. For the definition, choose any values $q\in[1,2)$ and $p\in(1,\infty]$ such that $1/p+1/q=1$ where $1/\infty=0$. Set $a_1=2$, $\delta_1=1$, and $\eta_1=1$ for any $q\in[1,2)$.   

If $q>1$, we define the sequences $(a_k)\subset(0,2]$ and $(\delta_k)\subset (0,1]$ by the formulas 
\begin{equation}\label{b1}
\frac{2p}{2+a_k}\left(1-\frac{\delta_{k+1}}{\delta_k}\right)-1+\left(\frac{\delta_{k+1}}{\delta_k}\right)^p=0 \text{\ \  with \ \  }\delta_{k+1}<\delta_k
\end{equation}
and 
\begin{equation}\label{b2}
a_{k+1}=2-\left(\frac{\delta_{k+1}}{\delta_k}\right)^{p-1}(2+a_k)
\end{equation}
for all $k\in \mathbb{N}$ by induction. We then put $\eta_k=\delta_k^p$  for all $k\in \mathbb{N}$. If $q=1$, we define the sequences $(a_k)\subset (0,2]$ and $(\eta_k)\subset (0,2]$ by 
\begin{equation}\label{b3}
\frac{2}{2+a_k}\log{\frac{\eta_{k}}{\eta_{k+1}}}-1+\frac{\eta_{k+1}}{\eta_{k}}=0 \text{\ \  with \ \ }\eta_{k+1}<\eta_k
\end{equation}
and 
\begin{equation}\label{b4}
a_{k+1}=2-\frac{\eta_{k+1}}{\eta_{k}}(2+a_k)
\end{equation}
for every $k\in \mathbb{N}$ by induction.  Moreover, we set $\delta_k=1$ for all $k\in \mathbb{N}$. Lastly, for any $q\in[1,2)$, we put $\tilde{\eta}_k=\eta_k^{1/q}$.

We remark that $(a_k)$ and $(\delta_k)$ if $q>1$ and $(a_k)$ and $(\eta_k)$ if $q=1$ are strictly decreasing and converge to zero as $k\to \infty$. Moreover, $\sum_{k=1}^\infty a_k=\infty$ for all $q\in[1,2)$. See Lemma \ref{lem:b2} below.  As we will see in Theorem \ref{th:sup1}  below, for each $k\in \mathbb{N}$, $a_k$ determines the limit profile and energy of the $k$-th concentration part and $\delta_k$ and $\eta_k$ describe the height of it. Finally, we remark  that, as noted in Subsection 1.3 in \cite{N3}, \eqref{b3} and \eqref{b4} can be regarded as the limit formulas of \eqref{b1} and \eqref{b2} in the limit $q\to1^+$, that is, $p\to \infty$. Hence, those limit formulas imply the differences and continuous relations between the two types of infinite concentration and oscillation behaviors observed for the cases $q\in(1,2]$ and $q=1$.

Lastly, using the sequence $(a_k)$, we define the infinite sequence $(z_k)$ of the limit profiles. For each $k\in \mathbb{N}$, we put
\begin{equation}\label{def:zk}
z_k(r)=\log{\frac{2a_k^2 b_k}{r^{2-a_k}(1+b_kr^{a_k})^2}}
\end{equation}
for all $r\ge0$ if $k=1$ and all $r>0$ if $k>1$ with $b_k=(\sqrt{2}/a_k)^{a_k}$. Then,  for every $k\in \mathbb{N}$, $z_k$ satisfies 
\begin{equation}\label{eq:zk}
\begin{cases}
-z_k''-\frac1r z_k'=e^{z_k}\text{ in }(0,\infty),\\
z_k(a_k/\sqrt{2})=0,\ (a_k/\sqrt{2})z_k'(a_k/\sqrt{2})=2,
\end{cases}
\end{equation}
and 
\[
\int_0^\infty e^{z_k}rdr=2a_k.
\]
We notice that $z_k$ has a singularity at the origin if and only if $k\ge2$. For every $k\in \mathbb{N}$, $z_k$ gives the limit profile of the $k$-th concentration part. 
\subsection{Infinite concentration estimates in \cite{N3}}\label{sub:mr}
Now let us state our concentration estimates in the first part \cite{N3}. Let $(\la_n,\mu_n,u_n)\in (0,\infty)\times (0,\infty)\times C^2([0,1])$ satisfy  
\begin{equation}\label{pn}
\begin{cases}
- u_n''-\frac1r u_n'=\la_n hf(u_n), \ u_n>0&\text{ in }(0,1),\\
u_n(0)=\mu_n,\ u_n'(0)=0=u_n(1),
\end{cases}
\end{equation}
for all $n\in \mathbb{N}$. We call $\{(\la,\mu_n,u_n)\}$ a sequence of solutions of \eqref{pn}. We get the next infinite concentration estimates for the sequence as follows.   
\begin{theorem}[\cite{N3}]\label{th:sup1} Assume (H1) with $q\in [1,2)$. Let $\{(\la_n,\mu_n,u_n)\}$ be any sequence of solutions of \eqref{pn} such that $\mu_n\to \infty$ as $n\to \infty$.  Then, extracting a subsequence if necessary, for each $k\in \mathbb{N}$, there exists a sequence $(r_{k,n})\subset (0,1)$ of values such that 
 $u_n(r_{k,n})/\mu_n\to \delta_k$
and 
\begin{equation}\label{eq:sup0}
\la_nr_{k,n}^2h(r_{k,n})f'(u_n(r_{k,n}))\to \frac{a_k^2}2
\end{equation}
as $n\to \infty$ and if we put sequences $(\ga_{k,n})$ of values and $(z_{k,n})$ of functions so that for each large $n\in \mathbb{N}$,
\[
\la_nh(r_{k,n})f'(u_n(r_{k,n}))\ga_{k,n}^2=1
\]
and 
\[
z_{k,n}(r)=g'(u_n(r_{k,n}))(u_n(\ga_{k,n}r)-u_n(r_{k,n}))
\]
for all $r\in [0,1/\ga_{k,n}]$, then  we have that $\ga_{k,n}\to0$ and there exist sequences $(\rho_{k,n}),(\bar{\rho}_{k,n})\subset [0,1)$ of values, where we chose $\bar{\rho}_{k,n}=0$ for all $n\in \mathbb{N}$ if $k=1$, such that $\rho_{k,n}\to0$, $\bar{\rho}_{k,n}/\ga_{k,n}\to0$, $\rho_{k,n}/\ga_{k,n}\to \infty$,  $u_n(\bar{\rho}_{k,n})/\mu_n\to \delta_k$, $u_n(\rho_{k,n})/\mu_n\to \delta_k$,  and
\[
\|z_{k,n}-z_k\|_{C^2_{\text{loc}}([\bar{\rho}_{k,n}/\ga_{k,n},\rho_{k,n}/\ga_{k,n}])}\to0,
\]
as $n\to \infty$ and further, if $q=1$, we get more precisely that 
\begin{equation}\label{eq:sup00}
\displaystyle u_n(r_n)=\mu_n-\left(\log{\frac1{\eta_k}}+o(1)\right)\frac{g(\mu_n)}{g'(\mu_n)}
\end{equation}
for the sequences $(r_n)=(r_{k,n}),(\rho_{k,n}),$ and $(\bar{\rho}_{k,n})$   as $n\to \infty$. Moreover, for all $k\in \mathbb{N}$,  we get 
\begin{equation}\label{eq:sup01}
\begin{split}
\lim_{n\to \infty}g'(u_n(r_{k,n}))\int_{\bar{\rho}_{k,n}}^{\rho_{k,n}}&\la_n hf(u_n)rdr\\
&= 2a_k=\lim_{n\to \infty}\int_{\bar{\rho}_{k,n}}^{\rho_{k,n}}\la_n hf'(u_n)rdr
\end{split}
\end{equation}
and 
\begin{equation}\label{eq:sup02}
\lim_{n\to \infty}g'(\mu_n)\int_{\rho_{k,n}}^{\bar{\rho}_{k+1,n}}\la_n hf(u_n)rdr= 0.
\end{equation}
Furthermore, we obtain that
\begin{equation}\label{eq:sup1}
\lim_{n\to \infty}\int_0^1\la_n hf'(u_n)rdr=\infty.
\end{equation}
Finally, additionally assuming (H2) if $q=1$, we have that
\begin{equation}\label{eq:sup2}
\lim_{n\to \infty}\frac{\log{\frac1{\la_n}}}{g(\mu_n)}=0
\end{equation}
and 
\begin{equation}\label{eq:sup3}
\lim_{n\to \infty}\frac{\log{\frac1{r_{k,n}}}}{g(\mu_n)}=\frac{\eta_k}{2}
\end{equation}
for all $k\in \mathbb{N}$ and all $q\in[1,2)$.
\end{theorem}
The former assertions are proved in Theorem 1.2 while the latter ones \eqref{eq:sup1}, \eqref{eq:sup2}, and \eqref{eq:sup3} follow from Theorem 1.3 in \cite{N3}. The former ones tell us that for each $k\in \mathbb{N}$, there exists a ``center" $(r_{k,n})$ of the $k$-th concentration such that, after appropriate scaling  around it, blow-up solutions converge to the limit profile  $z_k$. Moreover, we see that each asymptotic  ``energy" and ``height" are given by $2a_k$ and approximately by $\delta_k$ times the maximum value respectively where if  $q=1$, we have more precise information \eqref{eq:sup00}. Consequently, since $k$ is arbitrary, we observe an infinite sequence of concentrating parts. Moreover, as the infinite series of $(a_k)$ diverges to infinity, we observe that the total energy diverges to infinity as in \eqref{eq:sup1}. (We remark that the uniform estimate for the $L^1$ norm of $hf'(u)$ plays a key role in  the study of the oscillation property of the bifurcation diagram in the supercritical case. See Lemma 2.2 in  \cite{FZ} and Lemma 4.2 in \cite{Ku}.)

Here we recall the result by Kumagai \cite{Ku}. He shows that blow-up solutions converge to singular solutions up to  subsequences in the region away from the origin. See Theorem 1.1 (B) there. Hence we may conclude that any blow-up solutions behave as an infinite sequence of bubbles in the region near the origin while it behaves as singular solutions in the region away from the origin. This behavior is drastically different from those in the subcritical and critical cases observed in the previous works. (For the standard subcritical behavior in a disc, see Theorems 1.1 and A.1 in \cite{N3} for instance.) 
In this paper, we carry out the deeper analysis of the interaction between  blow-up solutions and singular solutions (which are not necessarily the limit one). 

We finally remark that there are striking differences between the cases $q>1$ and $q=1$. As noted above, the height of the center $(r_{k,n})$ of the $k$-th concentration is asymptotically given by $\delta_k$ times the maximum value in both cases. Here notice that $\delta_k$ is strictly decreasing sequence whose limit is zero if $q>1$ while $\delta_k=1$ for all $k\in \mathbb{N}$ if  $q=1$. This means that in the case $q=1$, all the concentrating parts appear  in the much higher (almost highest)  region than that in the case $q>1$. In other words, in the case $q=1$, the infinite concentration occurs in the region which is much closer to the origin.  As noted above, since  \eqref{b3} and \eqref{b4} can be obtained as the suitable limit of \eqref{b1} and \eqref{b2} respectively, we may regard this is the limit behavior of that in the case $q>1$.  
\subsection{Oscillation analysis in this paper}\label{sub:intosc}
Finally, we can give more precise explanations of the outline of our main results in this paper. Thanks to our concentration estimates above, especially using the explicit description by the limit profiles and combining the energy estimates  \eqref{eq:sup01} and \eqref{eq:sup02} with the Green type identity in Lemma \ref{lem:id} below, we describe the asymptotic shapes of the graphs of blow-up solutions near the origin in Theorem \ref{th:osc}. This is the most essential and original result in this paper which leads to all the subsequent conclusions. Indeed, this  shows that each concentration produces a bump whose  top is attained at the center $(r_{k,n})$ of the concentration and whose bottom  $(r_{k,n}^*)$ appears between the two successive concentrating parts. See Corollary \ref{cor:tb} below. Then, thanks to our precise quantification of the tops and bottoms there, we ensure that any blow-up solutions oscillate infinitely many times around singular functions satisfying a suitable asymptotic condition. See the condition (C) and  Proposition \ref{cor:o2} in Subsection \ref{sub:int}. Here, from the argument in \cite{Mi}, we know that the divergence of the intersection number between blow-up solutions and singular solutions  induces  infinite boundary (or zero point) oscillations of the corresponding shooting type problems. 
 Noting these facts and ideas, we provide a reasonable sufficient condition based on (C) for several infinite oscillation behaviors of blow-up solutions around singular solutions. See Theorem \ref{th:bif} and Corollary \ref{cor:bif} which contain the assertions for the infinite oscillation of the bifurcation diagram  and the existence of infinitely many solutions of \eqref{p}.  Finally, adapting our setting to that in \cite{FIRT} via the  additional condition (H3) introduced in Subsection \ref{sub:bd} below, we prove that the singular solutions constructed in \cite{FIRT} confirm the desired condition (C). Especially, we show  that the desired oscillation phenomena appear on \eqref{p} with certain classes of nonlinearities based on (H3) in Theorem \ref{th:bif2} and Corollary \ref{cor:bif3}. This completes all of our main results. In addition, see the discussion under Remark \ref{rmk:bifa} for the explanation for the climbing-up behavior of the infinite sequence of bubbles.  
\subsection{Organization}
The organization of this paper is the following. In Section \ref{sec:main}, we state all of our main results on infinite oscillation phenomena on \eqref{p}. Next in Section \ref{sec:pre}, we collect basic properties of our generalized nonlinearities, key identities, and some notes on the energy recurrence formulas which will be used in our proof. Next,  in Sections \ref{sec:osc}, we prove our main theorems on  oscillation and intersection properties of blow-up solutions.  Finally, in Section \ref{sec:bif}, we show the proof of our main results on infinite oscillations  of blow-up solutions around singular solutions.   

In this paper, we often extract subsequences from given sequences without any changes of suffixes for simplicity. 
\section{Main results: Infinite oscillation estimates}\label{sec:main}
In this section, we state all of our main theorems.  In the following, for a given sequence $\{(\la_n,\mu_n,u_n)\}$ of solutions  of \eqref{pn}, by extracting a subsequence if necessary, we let  $(r_{k,n}),(\ga_{k,n}),(\rho_{k,n}),(\bar{\rho}_{k,n})\subset (0,1)$ be the sequences of values in Theorem \ref{th:sup1} for all $k\in \mathbb{N}$ without specific comments. We also recall sequences $(a_k)$, $(b_k)$, $(\delta_k)$, $(\eta_k)$ of numbers defined in Subsection \ref{sub:bf}. 
\subsection{Oscillation estimates}\label{sub:osc}
The next theorem is the most essential and original result in this paper which describes the asymptotic shapes of the graphs of blow-up solutions near the origin.  
\begin{theorem}\label{th:osc} Assume (H1) and $\{(\la_n,\mu_n,u_n)\}$ is a sequence of solution of \eqref{pn} such that $\mu_n\to \infty$ as $n\to \infty$. Then, up to a subsequence, we have the following. 
\begin{enumerate}
\item[(i)] Suppose $1\le q< 2$ and   $k\in \mathbb{N}$.  Take any sequence $(r_n)\subset [\bar{\rho}_{k,n},\rho_{k,n}]$ and  put $R_n=r_n/\ga_{k,n}$ for all $n\in \mathbb{N}$. Then we obtain that
\[
\begin{split}
g&(u_n(r_n))+\log{g'\left(u_n(r_n)\right)}=2\log{\frac1{\sqrt{\la_n}r_n}}+\log{\frac{2a_k^2 b_k R_n^{a_k}}{ h(0)(1+b_kR_n^{a_k})^2}}+o(1)
\end{split}
\]  
as $n\to \infty$. Especially,  assuming in addition $g(\mu_n)^{-1}\log{R_n}\to0$ as $n\to \infty$ if $k=1$, we have that 
\[
g(u_n(r_n))=\left(2+o(1)\right)\log{\frac1{r_n}}
\] 
as $n\to \infty$ for all $k\in\mathbb{N}$ where we additionally supposed also  (H2) if $q=1$. 
\item[(ii)] Assume $1<q<2$ and  $k\in \mathbb{N}$. Then for any sequence $(r_n)\subset [\rho_{k,n},\bar{\rho}_{k+1,n}]$ and value $\delta\in[\delta_{k+1},\delta_k]$ such that $u_n(r_n)/\mu_{n}\to \delta$, we  have that
\[
g(u_n(r_n))=(\alpha_{q,k}(\delta)+o(1))\log{\frac1{\sqrt{\la_n}r_n}}=(\alpha_{q,k}(\delta)+o(1))\log{\frac1{r_n}}
\]
as $n\to \infty$ where $\alpha_{q,k}:[\delta_{k+1},\delta_k]\to (0,2]$ is a function defined by
\[
\alpha_{q,k}(x)=\frac{2\left(x/\delta_k\right)^p}{1-\frac{2p}{2+a_k}\left(1-x/\delta_k\right)}
\]
for all $x\in[\delta_{k+1},\delta_k]$. 
\item[(iii)] Suppose $q=1$ and $k\in \mathbb{N}$. Then for any sequence $(r_n)\subset [\rho_{k,n},\bar{\rho}_{k+1,n}]$ and value $\eta\in [\eta_{k+1},\eta_k]$ such that 
\[
u_n(r_n)=\mu_n-\left(\log{\frac1\eta}+o(1)\right)\frac{g(\mu_n)}{g'(\mu_n)}
\]
as $n\to \infty$, we get that
\[
g(u_n(r_n))=(\alpha_{1,k}(\eta)+o(1))\log{\frac1{\sqrt{\la_n}r_n}}
\]
as $n\to \infty$ where $\alpha_{1,k}:[\eta_{k+1},\eta_k]\to (0,2]$ is a function defined by
\[
\alpha_{1,k}(x)=\frac{2\left(x/\eta_k\right)}{1-\frac{2}{2+a_k}\log{(\eta_k/x)}}
\]
for all $x\in[\eta_{k+1},\eta_k]$. Particularly, assuming additionally (H2), we obtain that
\[
g(u_n(r_n))=(\alpha_{1,k}(\eta)+o(1))\log{\frac1{r_n}}
\]
as $n\to \infty$.
\end{enumerate}
\end{theorem}
We remark on (i). 
\begin{remark}\label{rmk:top} Noting the conclusion (i) with the fact
\[
\max_{R>0}\frac{2a_k^2 b_k R^{a_k}}{(1+b_kR^{a_k})^2}=\frac{a_k^2}{2}
\]
and \eqref{eq:sup0}, we can understand that the center  $(r_{k,n})$ of the $k$-th concentration attains the top of the $k$-th  bump. See the explanation below.
\end{remark}
We also give a comment on the function $\alpha_{q,k}$ in (ii) and (iii).
\begin{remark}\label{rmk:bot} Fix any $k\in \mathbb{N}$. Note that if $q\in(1,2)$, then the function $\alpha_{q,k}(x)$ in (ii) satisfies $\alpha_{q,k}(\delta_k)=\alpha_{q,k}(\delta_{k+1})=2$ by \eqref{b1} and has a unique minimum point $\delta_k^*\in(\delta_{k+1},\delta_k)$ with $\alpha_{q,k}^*:=\alpha_{q,k}(\delta_k^*)\in(0,2)$. Moreover, we see that $\alpha_{q,k}$ is strictly decreasing and increasing on $(\delta_{k+1},\delta_k^*)$ and $(\delta_k^*,\delta_k)$ respectively. We can explicitly compute that  
\[
\delta_k^*=\left(1-\frac{a_k}{2(p-1)}\right)\delta_k\text{\ \ \  and\ \ \  }\alpha_{q,k}^*=(2+a_k)\left(1-\frac{a_k}{2(p-1)}\right)^{p-1}.
\]
 Similarly, if $q=1$, then the function $\alpha_{1,k}$ in (iii) verifies $\alpha_{1,k}(\eta_k)=\alpha_{1,k}(\eta_{k+1})=2$ by \eqref{b3} and possesses a unique minimum point $\eta_k^*\in(\eta_{k+1},\eta_k)$ with $\alpha_{1,k}^*:=\alpha_{1,k}(\eta_k^*)\in(0,2)$. Furthermore,  $\alpha_{1,k}$ is strictly decreasing and increasing on $(\eta_{k+1},\eta_k^*)$ and $(\eta_k^*,\eta_k)$ respectively. We explicitly calculate that  
\[
\eta_k^*=e^{-a_k/2}\eta_k\text{\ \ \ and\ \ \ }\alpha_{1,k}^*=(2+a_k)e^{-a_k/2}. 
\]
Moreover, since  $(a_k)$ is strictly decreasing and converges to zero as $k\to \infty$ as in Lemma \ref{lem:b2} below, we can readily prove  that for any $q\in[1,2)$, the sequence $(\alpha_{q,k}^*)$ is strictly increasing and converges to $2$ as $k\to \infty$. These numbers characterize the bottom of the $k$-th bump as mentioned below.
\end{remark}
Noting the remarks above, we describe the precise estimates for the top and bottom of the $k$-th bump.
\begin{corollary}\label{cor:tb} Assume as in Theorem \ref{th:osc}. Then for each $k\in \mathbb{N}$, we have that
\[
\begin{split}
g&(u_n(r_{k,n}))+\log{g'\left(u_n(r_{k,n})\right)}=2\log{\frac1{\sqrt{\la_n}r_{k,n}}}+\log{\frac{a_k^2 }{2 h(0)}}+o(1)
\end{split}
\]
as $n\to \infty$ up to a subsequence. Moreover, there exists a sequence $(r_{k,n}^*)\subset (r_{k,n},r_{k+1,n})$ such that 
\[
\begin{cases}
\displaystyle\frac{u_n(r_{k,n}^*)}{\mu_n}\to \delta_k^* &\text{ if $q>1$},\\
\displaystyle u_n(r_{k,n}^*)=\mu_n-\left(\log{\frac1{\eta_k^*}}+o(1)\right)\frac{g(\mu_n)}{g'(\mu_n)}&\text{ if $q=1$,}  
\end{cases}
\]
and
\[
g(u_n(r_{k,n}^*))=\left(\alpha_{q,k}^{*}+o(1)\right)\log{\frac1{\sqrt{\la_n}r_{k,n}^*}}
\] 
as $n\to \infty$ up to a subsequence. In particular, additionally assuming (H2) if $q=1$, we get that
\[
g(u_n(r_{k,n}))=(2+o(1))\log{\frac1{r_{k,n}}}
\]
and
\[
g(u_n(r_{k,n}^*))=\left(\alpha_{q,k}^{*}+o(1)\right)\log{\frac1{r_{k,n}^*}}
\]
as $n\to \infty$ for all $q\in [1,2)$ up to a subsequence.
\end{corollary}
For example, if $g(t)=t^p$ with $p\in(2,\infty)$, then the final formulas  imply
\[
u_n(r_{k,n})=\left((2+o(1))\log{\frac1{r_{k,n}}}\right)^{\frac1p}
\]
and
\[
u_n(r_{k,n}^*)=\left((\alpha_{q,k}^*+o(1))\log{\frac1{r_{k,n}}}\right)^{\frac1p}
\]
as $n\to \infty$ for all $k\in \mathbb{N}$ where $q=p/(p-1)$ and if $g(t)=\exp_l{(t^m)}$ with  $l\in \mathbb{N}$, $l\ge2$, and $m>0$, then
\[
u_n(r_{k,n})=\left\{\log_l{\left((2+o(1))\log{\frac1{r_{k,n}}}\right)}\right\}^{\frac1m}
\]
and
\[
u_n(r_{k,n}^*)=\left\{\log_l{\left((\alpha_{1,k}^*+o(1))\log{\frac1{s_{k,n}}}\right)}\right\}^{\frac1m}
\]
as $n\to \infty$ for all $k\in \mathbb{N}$ where we defined $\log_1{(t)}=\log{(t)}$ and $\log_l{(t)}=\log_{l-1}{(\log{t})}$ for all $l\in\mathbb{N}$ with $l\ge2$ by induction for all large $t>0$.

The theorem above gives the desired estimates for wider and new classes of  nonlinearities than those  in  Theorem 5.1 in \cite{N2}.  Particularly, the conclusions for the multiple exponential case $q=1$ are completely new. Recall also that even for the Trudinger-Moser type case $q\in(1,2)$, (H1) admits the new cases as noted in Subsection \ref{sub:ge}. The basic qualitative  interpretations for the cases $q>1$ and $q=1$ are same while there is the exact quantitative difference described with the two different  types of the energy recurrence formulas defined in Subsection \ref{sub:bf}. 

First, (i) implies that for each $k\in \mathbb{N}$, $u_n(r)$ stays on the curve approximately given by $g^{-1}((2+o(1))\log(1/r))$, where $g^{-1}$ is the inverse function of $g$, as long as $r$ stays in the $k$-th concentration interval $[\bar{\rho}_{k,n},\rho_{k,n}]$. Moreover by Remark \ref{rmk:top}, we notice that it attains the highest curve  at the center $(r_{k,n})$ of the $k$-th concentration. Next, we observe with (ii) and (iii) that as $r$ increases from $\rho_{k,n}$ to $r_{k,n}^*$, $u_n(r)$ rapidly goes down and crosses the lower curve $g^{-1}((\alpha_{q,k}(x)+o(1))\log{(1/r)})$ for all $x\in[\delta_k^*,\delta_k]$ if $q>1$ and for all $x\in[\eta_k^*,\eta_k]$ if $q=1$. Then, it finally arrives at the lowest curve $g^{-1}((\alpha_{q,k}^*+o(1))\log{(1/r)})$ at $r=r_{k,n}^*$. After that, as $r$ increases from $r_{k,n}^*$ to $r_{k+1,n}$, this time $u_n(r)$  goes down much more slowly and it crosses the upper curve $g^{-1}((\alpha_{q,k}(x)+o(1))\log{(1/r)})$ for all $x\in[\delta_{k+1},\delta_k^*]$ if $q>1$ and for all $x\in[\eta_{k+1},\eta_k^*]$ if $q=1$. Especially, it lastly arrives at the highest curve approximately given by $g^{-1}((2+o(1))\log{(1/r)})$ at the center $(r_{k+1,n})$ of the $(k+1)$-th concentration. This means the $u_n(r)$ relatively ``goes up" in the interval $[r_{k,n}^*,r_{k+1,n}]$.  Since $k$ is arbitrary, these down-and-up behaviors produce an infinite sequence of  bumps on the graphs of large solutions near the origin.  In other words, we prove an infinite oscillation behavior  of blow-up solutions among the suitable  asymptotic curves.  Notice also that since $\alpha_{q,k}^*$ is strictly increasing and converges to $2$ as $k\to \infty$,  the ``amplitude" of the $k$-th oscillation becomes smaller and smaller as $k$ increases and finally vanishingly small as $k\to \infty$. 

Now, we note the quantitative differences between the cases $q>1$ and $q=1$. That is,  in the case $q=1$, since the infinite concentration occurs in a much higher (the almost highest) region as noted before, so does the infinite oscillation. In this sense, a much rapider infinite oscillation is observed in the multiple exponential case $q=1$. 

Finally, we remark that this result widely generalizes and more precisely quantifies the discussion on ``the bouncing process"  in \cite{McMc}. Thanks to our  precise quantification, we can proceed to the analysis of  intersection properties as follows.  
\subsection{Intersection properties}\label{sub:int}
The previous oscillation estimates allow us to show infinite intersection properties between blow-up solutions and singular functions with suitable asymptotic behaviors. To see this, for any values $\alpha\in(0,2]$ and $\beta\in \mathbb{R}$, we put a function
\[
U_{\alpha,\beta}(r)=\alpha\log{\frac1r}+\beta
\]
for all small $r>0$. Then we introduce the next key condition.  We say a  continuous function $U(r)$ defined for all small $r>0$ satisfies the condition (C) if   
\begin{enumerate}
\item[(C)] for any $\alpha\in(0,2)$ and $\beta\in \mathbb{R}$, there exists a value $\bar{r}>0$ such that 
\[
U_{\alpha,0}(r)\le g(U(r))\text{\ \ \  and \ \ \ }g(U(r))+\log{g'(U(r))}\le U_{2,\beta}(r)
\]
for all $r\in (0,\bar{r})$.
\end{enumerate} 
Then, we give the next assertion. In the following, for any interval $I\subset (0,\infty)$ and continuous function $u$ on $I$, we define $Z_{I}[u]$ as the number of zero points of $u$ on $I$. 
\begin{proposition}\label{cor:o2} Assume (H1), (H2), and $\{(\la_n,\mu_n,u_n)\}$ is a sequence of solutions of \eqref{pn} such that  $\mu_n\to \infty$ as $n\to \infty$. Moreover, suppose $U(r)$ is a continuous function defined for all small $r>0$ satisfying the condition (C).  Then for all  $k\in \mathbb{N}$, there exist sequences $(r_{k,n}^\pm)\subset (0,1)$ such that $r_{k,n}^-<r_{k,n}<r_{k,n}^+$ and $u_n(r_{k,n}^\pm)=U(r_{k,n}^\pm)$ for all large $n\in \mathbb{N}$ and $u_n(r_{k,n}^\pm)/\mu_n\to \delta_k$ and if $q=1$, 
\[
u_n(r_{k,n}^{\pm})=\mu_n-\left(\log{\frac1{\eta_k}}+o(1)\right)\frac{g(\mu_n)}{g'(\mu_n)}
\]
as $n\to \infty$ up to a subsequence. Especially, for any sequence $(r_n)\subset (0,1)$, in the interval where $U$ is defined,  such that $u_n(r_n)/\mu_n\to0$ if $q>1$ and $(\mu_n-u_n(r_n))g'(\mu_n)/g(\mu_n)\to \infty $ if $q=1$ as $n\to \infty$, we get 
\[
\lim_{n\to \infty}Z_{(0,r_n)}[u_n-U]=\infty.
\]
\end{proposition}
In this proposition, we observe that the condition (C)  gives a sufficient condition on the asymptotic behavior of given singular functions $U$ for the divergence of the intersection number between blow-up solutions $u_n$ and $U$. We remark that this corollary shows not only the divergence of intersection numbers but also the precise estimates for the positions where the intersections occur. In particular, it tells us that two intersections occur at the points $(r_{k,n}^\pm)$ near the center $(r_{k,n})$ of the $k$-th concentration for all $k\in \mathbb{N}$. This implies that in the case $q=1$, the infinite intersection appears in much higher region than that in the case $q\in(1,2)$. Hence we again observe two types of infinite intersection properties depending on the choice of the growth. 

We here note that  for the application to the analysis of infinite oscillations around singular solutions in the next subsection,  we prove the similar result on the corresponding shooting type problem. See Proposition \ref{cor:int2} below. Based on these intersection assertions, we  proceed to our final discussion. 
\subsection{Infinite oscillations around singular solutions}\label{sub:bd}
We lastly provide several results on the infinite oscillation of blow-up solutions around singular solutions. In the following, we focus on the case $f\in C^2([0,\infty))$ with $f>0$ on $(0,\infty)$ and $h\equiv1$. In this case, from Lemma 2.1 in \cite{AKG} with Theorem 2.1 in \cite{NT}, for each $\mu>0$, there exists a unique solution $(\la,u)=(\la(\mu),u(\mu,\cdot))$ of \eqref{p} with $u(\mu,0)=\mu$. Then, from Theorem 2.1 in \cite{K}, the map $\mu \mapsto (\la(\mu),u(\mu,\cdot))$, defined for all $\mu>0$, draws a $C^1$-curve in $\mathbb{R}\times C^{2,\alpha}([0,1])$ which covers all the solutions $(\la,u)$ of \eqref{p}. Our aim is to study the infinite oscillation behavior along this solutions curve under our setting (H1). To this end,  it is important to discuss the  intersection property between blow-up solutions and singular solutions, that is,  solutions $(\la,U)\in (0,\infty)\times C^2((0,1])$ of the next problem, 
\begin{equation}\label{eq:sg}
\begin{cases}
-U''-\frac 1r U'=\la f(U),\ U>0\text{ in }(0,1),\\
\lim_{r\to 0^+}U(r)=\infty,\ U(1)=0.
\end{cases}
\end{equation}
As we can expect from the previous proposition, the condition (C)  gives a reasonable sufficient condition for the divergence of the intersection number between blow-up solutions and singular solutions. Hence, with the idea in \cite{Mi}, we see that it turns into a natural sufficient condition for the several infinite oscillation behaviors around singular solutions along the solutions curve as follows. 
\begin{theorem}\label{th:bif} Assume (H1) with $q\in[1,2)$, $f\in C^2([0,1])$ with $f>0$ on $(0,\infty)$, and $h\equiv1$. If $q=1$, we suppose in addition (H2). Moreover, assume that there exists a solution $(\la,U)=(\la^*,U^*)$ of \eqref{eq:sg} such that $U^*$ satisfies the condition (C).  Finally, let $(\la(\mu),u(\mu,\cdot))$ be the solutions $C^1$-curve of \eqref{p} as above. Then we get the following assertions (i), (ii), and (iii). 
\begin{enumerate}
\item[(i)]  $\la(\mu)$  oscillates around $\la^*$ infinitely many times as $\mu\to \infty$.
\item[(ii)] $u_r(\mu,1)$ oscillates around $(U^*)'(1)$ infinitely many times as $\mu\to \infty$  where $u_r(\mu,r)$ denotes the first derivative of $u(\mu,r)$ with respect to $r$.
\item[(iii)] If we additionally assume (H2),
then we get that
\[
\lim_{\mu \to \infty}Z_{(0,1)}[u(\mu,\cdot)-U^*]=\infty.
\]
\end{enumerate} 
\end{theorem}
\begin{remark} The oscillation assertions in (i) and (ii) mean that there exist sequences $(\mu_n^\pm)$ and $(\nu_n^\pm)$ of positive values such that $\mu_n^\pm,\nu_n^\pm\to \infty$ as $n\to \infty$ and $\la(\mu_n^-)<\la^*<\la(\mu_n^+)$ and $|u_r(\nu_n^-,1)|<|(U^*)'(1)|<|u_r(\nu_n^+,1)|$ for all $n\in \mathbb{N}$. 
\end{remark}
The assertions (i), (ii), and (iii) above give three kinds of oscillation behaviors along the  solutions curve. Among others, (i) shows the infinite oscillation of the bifurcation diagram which leads to  the existence of infinitely many solutions as follows. 
\begin{corollary}\label{cor:bif}
Assume as in the previous theorem. Then, for any number $N\in \mathbb{N}$, there exists a value $\e>0$ such that \eqref{p} permits at least $N$ distinct solutions $u$ for all $\la^*-\e<\la<\la^*+\e$ and there exists an infinite sequence $(u_n)$ of solutions of \eqref{p} with $\la=\la^*$ such that $u_n(0)\to \infty$ as $n\to \infty$.
\end{corollary}
Moreover, (ii) and (iii) above describe interesting oscillation behaviors of the graph of the solution  $u(\mu,\cdot)$ around that of the singular solution $U^*$ for large $\mu>0$. (ii) gives the oscillation of the derivative at the boundary and (iii) is the oscillation of the graph proved by the infinite concentration behavior around the origin.

We note that the previous theorem is a consequence of the analysis of the corresponding shooting type problem \eqref{pv} below. We will see that (i) and (ii) are the direct consequences of the two kinds of the infinite boundary (or zero point) oscillations (i)' and (ii)' in Lemma \ref{lem:bif2}. An interesting remark is that from the view point of \eqref{pv}, we obtain a reasonable global picture of  our infinite concentration and oscillation phenomena as the climbing-up  behavior of an infinite sequence of bumps from the bottom to the top. We refer the reader to the note  under Remark \ref{rmk:bifa} for the detail. 
     
We lastly remark on some additional facts on the asymptotic behavior of $(\la(\mu),u(\mu,\cdot))$ as $\mu\to \infty$ based on  the results in our series and \cite{Ku}. First, our infinite concentration estimate  \eqref{eq:sup1} yields the divergence of the energy
\begin{equation}\label{eq:diven}
\lim_{\mu\to \infty}\int_0^1 \la(\mu)f'(u(\mu,r))rdr=\infty.
\end{equation}
We point out that this is implicitly shown in the proof of (D)  of Theorem 1.1 in \cite{Ku} in view of the key estimates Lemma 4.2 and Proposition 4.3 there. Moreover, one notices that Lemma 4.2 there with the fact \eqref{eq:diven} directly proves the divergence of the Morse index of $u(\mu,\cdot)$ as $\mu\to \infty$. Hence we may say our approach gives another proof of the key fact \eqref{eq:diven} via the construction of the infinite sequence of bubbles. Moreover, we note that (B) of Theorem 1.1 shows the convergence of $(\la(\mu_n),u(\mu_n,\cdot))$ to a singular solution for a suitable sequence $(\mu_n)$ of positive values.  See Theorem 1.1 in \cite{Ku} for the detail. We also refer the reader to Lemma 2.13 in \cite{N3} for the connection between our setting and that in \cite{Ku}. 

Now, let us show an application of the previous theorem. Then our aim becomes to ensure the existence of solutions $(\la,U)$ of \eqref{eq:sg}  such that $U$ satisfies the condition (C).  Fortunately, the existence of such solutions is recently proved by \cite{FIRT} under a related generalized setting. To adapt our setting to  ($f_1$), $(f_2)$, $(f_3)$, and Lemma 4.1 in that  paper, we introduce the next condition. 
\begin{enumerate}
\item[(H3)] (H1) holds true with $q\in[1,2)$. Moreover, $g\in C^5([t_0,\infty))$, $g'''(t)>0$ for all $t\ge t_0$, and there exists a value $\hat{q}\in[0,\infty]$ such that 
\begin{equation}\label{g3}
\lim_{t\to \infty}\frac{g''(t)^2}{g'(t)g'''(t)}=\hat{q}
\end{equation}
and
\begin{equation}\label{g4}
\lim_{t\to \infty}\frac{g''''(t)}{g''(t)^2}=0.
\end{equation}
Moreover, we have that  
\begin{equation}\label{g5}
\lim_{t\to \infty}\sqrt{g(t)}\left(q-\frac{g'(t)^2}{g(t)g''(t)}\right)=0,
\end{equation}
\begin{equation}\label{g7}
\lim_{t\to \infty}\sqrt{g(t)}\left(\hat{q}-\frac{g''(t)^2}{g'(t)g'''(t)}\right)=0,
\end{equation}
and
\begin{equation}\label{g6}
\lim_{t\to \infty}\sqrt{g(t)}\frac{g''''(t)}{g''(t)^2}=0.
\end{equation}
In addition, $g$ satisfies all the assumption for the function $a$ in Lemma 4.1 in \cite{FIRT}. 
\end{enumerate}
We will  see that $\hat{q}=(2-q)^{-1}$ by Lemma \ref{lem:hatq} below. The former assumptions \eqref{g3} and \eqref{g4} confirm  their ($f_1$) and ($f_2$) and the latter (stronger) ones \eqref{g5}, \eqref{g7}, and \eqref{g6} give $(f_3)$. Under this condition, we obtain the desired oscillation and existence results.  
\begin{theorem}\label{th:bif2} Assume (H3), $f\in C^2([0,\infty))$ with $f>0$ on $(0,\infty)$, and $h\equiv 1$. If $q=1$, we additionally suppose  (H2). 
 Then, there exists a solution $(\la,U)=(\la^*,U^*)$ of \eqref{eq:sg} such that $U^*$ satisfies the condition (C). In particular, letting $(\la(\mu),u(\mu,\cdot))$ be the solutions $C^1$-curve of \eqref{p} as above, we have that all the conclusions  (i), (ii), and (iii) in Theorem \ref{th:bif} hold true with these $(\la(\mu),u(\mu,\cdot))$ and $(\la^*,U^*)$. 
\end{theorem}
\begin{corollary}\label{cor:bif3} Assume as in the previous theorem. Then, the same conclusion with that in Corollary \ref{cor:bif} holds true. 
\end{corollary}
We can easily check that if  $f\in C^2([0,\infty))$ with $f>0$ on $(0,\infty)$ satisfies (H0) with $h\equiv 1$ and 
\begin{enumerate}
\item[(a)] $f(t)=t^me^{t^p+ct^{\bar{p}}}$ for any given constants $c,m\in \mathbb{R}$, $p>2$, and $0<\bar{p}<p/2$ or 
\item[(b)] $f(t)=e^{e^{t}+w(t)}$ where $w(t)$ is any given polynomial function
\end{enumerate}
for all $t\ge t_0$, then all the assumptions of the previous theorem are satisfied. In fact, the former conditions are confirmed by direct computations. (We refer the reader to Lemma 2.11 in \cite{N3} for a useful condition to check (H1) with $q=1$.) For the final condition concerning Lemma 4.1 in \cite{FIRT}, we give a note in Remark \ref{rmk:41} below for the reader's convenience. Consequently, we get  all the assertions in  the theorem and corollary for both examples. The Trudinger-Moser type growth (a) in the cases $m=0$ with $c=1$ and $m\not=0$ with $c=0$ is treated in  \cite{N2} (See Remark 5.11 there) while multiple exponential case (b) is completely new. To our best knowledge, this paper is the first work which shows the infinite oscillation of the bifurcation diagram around a suitable value and the existence of infinitely many solutions of \eqref{p} with the multiple exponential growth.  

Moreover, we emphasize that our sufficient condition in Theorem \ref{th:bif} is applicable to the other  new cases,  for example, the Trudinger-Moser type case $f(t)=e^{t^p(\log{t})^l}$ and the multiple exponential one $f(t)=\exp_k{(t^m(\log{t})^l)}$ noted in Subsection \ref{sub:ge}. However, at the moment, the assumptions \eqref{g5}, \eqref{g7}, and \eqref{g6}, i.e., $(f_3)$ in \cite{FIRT}, restrict the examples which allow the construction of singular solutions in \cite{FIRT}. See Remark 2.2 and Section 4 in \cite{FIRT} for their explanations for the examples. On the other hand,  the existence  of singular solutions is confirmed (without asymptotic formulas) under the general setting in \cite{Ku}. See Theorem 1.1 there. We believe that those singular solutions satisfy the  asymptotic condition (C) and thus, Theorem \ref{th:bif} should enable us to show the desired  assertions under our basic condition (H1). The answer is left for our future works.   

The rest of this paper is devoted to the proof of the theorems above.   
\section{Preliminaries}\label{sec:pre}
For the proof of our main theorem, we recall some properties of our generalized nonlinearity, key identities, and notes on the energy recurrence formulas proved in  \cite{N3}.
\subsection{Properties of generalized exponential growth}\label{sub:geg2}
We first collect all the basic properties of our generalized nonlinearity. All the  proofs are given in Subsection 2.1 in \cite{N3}. The next lemma shows the monotonicity of $g$ and $g'$ and also the basic asymptotic behavior of them.
\begin{lemma}\label{lem:monog} Assume (H1). Then we have that  $g(t)$ and $g'(t)$ are strictly increasing  for all $t\ge t_0$ and that $g(t)\to \infty$ and $g'(t)\to \infty$ as $t\to \infty$.  Moreover, we get
\begin{equation}\label{eq:lgg}
\lim_{t\to \infty}\frac{\log{g'(t)}}{g(t)}=0.
\end{equation}
\end{lemma}
The next lemma proves that $q$ is the H\"{o}lder conjugate exponent of $p$. 
\begin{lemma}\label{lem:pq} We suppose (H1). Then, we have that $1/p+1/q=1$ where we regarded $1/\infty=0$.
\end{lemma}
The next two lemmas are keys for the scaling argument in \cite{N3}. The latter assertion of (ii) in (H1) is used here.   
\begin{lemma}\label{lem:g2} Assume (H1). Then, for any value $M>0$, we get that 
\[
\lim_{t\to \infty}\sup_{-M\le y\le M}\left|\frac{g'\left(t+\frac{y}{g'(t)}\right)}{g'(t)}-1\right|=0.
\]
\end{lemma}
\begin{lemma}\label{lem:g1} 
 Suppose (H1). Then for any value $M>0$, we obtain that
\[
\lim_{t\to \infty}\sup_{-M\le y\le M} \left|g\left(t+\frac{y}{g'(t)}\right)-g(t)-y\right|=0.
\]
In particular, we have that
\[
\lim_{t\to \infty}\sup_{-M\le y\le M} \left|\frac{f\left(t+\frac{y}{g'(t)}\right)}{f(t)}-e^{y}\right|=0=\lim_{t\to \infty}\sup_{-M\le y\le M} \left|\frac{f'\left(t+\frac{y}{g'(t)}\right)}{f'(t)}-e^{y}\right|.
\]
\end{lemma}
We give the next two lemmas  used for computing the ratio of the growth $g$ at two blowing-up points. 
 \begin{lemma}\label{lem:g3} Assume (H1) with $q>1$ and take any sequences $(s_n),(t_n)\subset (0,\infty)$ of numbers and a value $x\in [0,1]$ such that $s_n<t_n$ for all $n\in \mathbb{N}$ and $s_n\to \infty$ and  $s_n/t_n\to x$ as $n\to \infty$. Then we have that
\[
\lim_{n\to \infty}\frac{g(s_n)}{g(t_n)}=x^p=\lim_{n\to \infty}\left(\frac{g'(s_n)}{g'(t_n)}\right)^q.
\]
\end{lemma}
\begin{lemma}\label{lem:g4} Suppose (H1) with $q=1$. Then, for any sequences  $(t_n),(x_n)\subset (0,\infty)$ of values and a number  $x_0\in [0,\infty)$ such that $t_n\to \infty$ and $x_n\to x_0$ as $n\to \infty$,  we have that
\[
\lim_{n\to \infty}\frac{g\left( t_n-x_n\frac{g(t_n)}{g'(t_n)}\right)}{g(t_n)}=e^{-x_0}=\lim_{n\to \infty}\frac{g'\left( t_n-x_n\frac{g(t_n)}{g'(t_n)}\right)}{g'(t_n)}.
\]
\end{lemma}
The former condition on (ii) in (H1) is applied for the proof of in the previous lemma. Finally, we give the next consequence of (H2).
\begin{lemma}\label{lem:kap} (H2) implies that
\[
\sup\{\la>0\ |\ \text{\eqref{p} admits a solution $u$}\}<\infty.
\]
\end{lemma}
\subsection{Identities}
We next recall some key identities and a consequence of the Pohozaev identity. For the proofs, see  Subsection 2.2 in \cite{N3}. In the following, let $\{(\la_n,\mu_n,u_n)\}$ be a sequence of solutions of \eqref{pn}. The first one is standard.
\begin{lemma} 
We have that
\begin{equation}\label{id0}
-ru_n'(r)=\int_0^r \la_nhf(u_n(s))sds
\end{equation}
for all $r\in[0,1]$ and all $n\in \mathbb{N}$. 
\end{lemma}
The next Green type identity is often used for the concentration analysis in \cite{N3}. This is also a key to deduce our oscillation estimates Theorem \ref{th:osc} above.
\begin{lemma}\label{lem:id} We get that for all $0<r\le 1$,
\begin{equation}\label{id1}
u_n(0)-u_n(r)=\int_0^r \la_nh(s)f(u_n(s))s\log{\frac rs}ds
\end{equation}
and for all $0<r<s\le1$,
\begin{equation}\label{id2}
\begin{split}
u_n(r)-u_n(s)=\log{\frac sr}&\int_0^r \la_nh(t)f(u_n(t))tdt\\
&+\int_r^s \la_nh(t)f(u_n(t))t\log{\frac st}dt
\end{split}
\end{equation}
for all $n\in \mathbb{N}$.
\end{lemma}
We also  give a consequence of the Pohozaev identity. (See Lemmas 2.16 and 2.17 in \cite{N3}.) 
\begin{lemma}\label{lem:gl}
Assume $\mu_n\to \infty$ as $n\to \infty$. Then $(u_n)$ is locally uniformly bounded in $(0,1]$. In particular, for any sequence $(r_n)\subset (0,1)$ such that $u_n(r_n)\to\infty$ as $n\to \infty$,  we have that $r_n\to0$ as $n\to \infty$. 
\end{lemma}
\subsection{Notes on energy recurrence formulas} Lastly we note some properties of the energy reccurence formulas in Subsection \ref{sub:bf}. Take any numbers $q\in[1,2)$ and $p\in (2,\infty]$ such that $1/p+1/q=1$ where $1/\infty=0$ and define sequences $(a_k)$, $(\delta_k)$, $(\eta_k)$, and $(\tilde{\eta}_k)$ of values as in Subsection \ref{sub:bf}. We first remark that by Lemma 2.20 in \cite{N3}, those sequences  are well-defined. Moreover, they have the next properties. Proofs are given by Lemmas 2.21 and 2.22 in \cite{N3}.
\begin{lemma}\label{lem:b2} If $q\in(1,2)$, then the sequences $(a_k)$ and $(\delta_k)$ are strictly decreasing and  converge  to zero as $k\to \infty$ and $\sum_{k=1}^\infty a_k=\infty$. If $q=1$, then the sequences $(a_k)$ and $(\eta_k)$ are strictly decreasing and converge to zero as $k\to \infty$ and $\sum_{k=1}^\infty a_k=\infty$. 
\end{lemma}
\begin{lemma}\label{lem:b3} Assume $q\in[1,2)$. Then we have that
\[
\displaystyle \tilde{\eta}_k\sum_{i=1}^k \frac{2a_i}{\tilde{\eta}_i}=2+a_k
\]
for all $k\in \mathbb{N}$.
\end{lemma} 
\section{Oscillation and intersection estimates}\label{sec:osc}
Let us start our  main argument.  In this section we complete our oscillation and intersection estimates in Subsections \ref{sub:osc} and \ref{sub:int}.  In this section, we always assume (H1) with $q\in[1,2)$. For all $k\in \mathbb{N}$, let $(r_{k,n}),$ $(\ga_{k,n}),(\rho_{k,n}),(\bar{\rho}_{k,n})\subset [0,1)$ be sequences of values in Theorem \ref{th:sup1} and write $\mu_{k,n}=u_n(r_{k,n})$ for all $n\in \mathbb{N}$.

\subsection{Proof of Theorem \ref{th:osc}}
We begin with the proof of  Theorem \ref{th:osc}. To this end, we first recall the next fact noted in Remark 5.1 in \cite{N3}. 
\begin{lemma}\label{lem:o0} We have that for all $k\in \mathbb{N}$ , 
\[
 \lim_{n\to \infty}\frac{\log{\frac{\rho_{k,n}}{\ga_{k,n}}}}{g(\mu_{k,n})}=0,\ \ \ \lim_{n\to \infty}\frac{\log{\frac{\bar{\rho}_{k,n}}{\ga_{k,n}}}}{g(\mu_{k,n})}=0,
\]
where we assumed $k\not=1$ in the latter formula, and
\[
\lim_{n\to \infty}\sup_{r\in [\bar{\rho}_{k,n}/\ga_{k,n}, \rho_{k,n}/\ga_{k,n}]}\left|\frac{(\ga_{k,n}r)^2 \la_nh(\ga_{k,n}r)f'(u_n(\ga_{k,n}r))}{r^2e^{z_k(r)}}-1\right|=0.
\]
\end{lemma}
Then, we  show (i) of the theorem.   
\begin{lemma}\label{lem:o4} Suppose $k\in \mathbb{N}$. Choose any sequence $(r_n)\subset [\bar{\rho}_{k,n},\rho_{k,n}]$ and set $R_n=r_n/\ga_{k,n}$ for all $n\in \mathbb{N}$. Then we have that
\[
\la_n r_n^2h(r_n) f'(u_n(r_n))=(1+o(1))\frac{2a_k^2b_kR_n^{a_k}}{(1+b_k R_n^{a_k})^2}
\]  
as $n\to \infty$. In particular, assuming $g(\mu_n)^{-1}\log{R_n}\to0$ as $n\to \infty$ if $k=1$, we obtain that 
\[
g(u_n(r_n))=(2+o(1))\log{\frac1{r_n}}
\]
as $n\to \infty$ for any $k\in \mathbb{N}$ where we additionally assumed  (H2) if $q=1$. 
\end{lemma}
\begin{proof} From the latter assertion of the previous lemma, we get that 
\[
\begin{split}
\la_n r_n^2h(r_n)f'(u_n(r_n))
&=(1+o(1))R_n^2e^{z_k(R_n)}
\end{split}
\]
as $n\to \infty$. This proves the former conclusion. Then since $g(u_n(r_n))/g(\mu_n)= \eta_k+o(1)$ as $n\to \infty$ by Lemmas \ref{lem:g3} and \ref{lem:g4}, using \eqref{eq:sup2}, \eqref{eq:lgg}, and the former formulas of the previous lemma, we prove the latter one. 
 This completes the proof. 
\end{proof}
For the cases (ii) and (iii), we use the next lemma.
\begin{lemma}\label{lem:o01}We have that
\[
\lim_{n\to \infty}\frac{\log{\frac{1}{\sqrt{\la_n}r_{k,n}}}}{g(\mu_n)}=\frac{\eta_k}{2}.
\]
\end{lemma}
\begin{proof} From \eqref{eq:sup0} and \eqref{eq:lgg}, we have that
\[
\lim_{n\to \infty}\frac{\log{\frac{1}{\sqrt{\la_n}r_{k,n}}}}{g(\mu_{k,n})}=\frac{1}{2}.
\]
Then Lemmas \ref{lem:g3} and \ref{lem:g4} give the desired formula. We finish the proof.
\end{proof}
We proceed to the proof for the case (ii). 
\begin{lemma}\label{lem:o5} Assume $1<q<2$ and $k\in \mathbb{N}$. Choose any sequence $(r_n)\subset [\rho_{k,n},\bar{\rho}_{k+1,n}]$ and value $\delta\in[\delta_{k+1},\delta_k]$ such that  $u_n(r_n)/\mu_n\to \delta$ as $n\to \infty$. Then we have that 
\[
g(u_n(r_n))=(\alpha_{q,k}(\delta)+o(1))\log{\frac1{\sqrt{\la_n}r_n}}=(\alpha_{q,k}(\delta)+o(1))\log{\frac1{r_n}}
\]
as $n\to \infty$.
\end{lemma}
\begin{proof} From \eqref{id2}, we observe that
\[
\log{\frac{r_n}{\rho_{k,n}}}\int_0^{\rho_{k,n}}\la_nhf(u_n)rdr \le u_n(\rho_{k,n})-u_n(r_n)\le \log{\frac{r_n}{\rho_{k,n}}}\int_0^{r_n}\la_nhf(u_n)rdr
\]
for all $n\in \mathbb{N}$. Then, \eqref{eq:sup01} and \eqref{eq:sup02} with Lemmas \ref{lem:g3} and \ref{lem:b3} show that 
\[
\begin{split}
g'(\mu_n)&(u_n(\rho_{k,n})-u_n(r_n))\\
&=\left(\frac{2+a_k}{\delta_k^{p-1}}+o(1)\right)\left(\log{\frac{r_{k,n}}{\rho_{k,n}}}+\log{\frac{1}{\sqrt{\la_n}r_{k,n}}}-\log{\frac1{\sqrt{\la_n}r_n}}\right)\\
&=\left(\frac{(2+a_k)\delta_k}{2}+o(1)\right)g(\mu_n)-\left(\frac{2+a_k}{\delta_k^{p-1}}+o(1)\right)\log{\frac1{\sqrt{\la_n}r_n}} 
\end{split}
\]
as $n\to \infty$ where for the latter equality we used Lemmas \ref{lem:o0} and \ref{lem:o01} with the fact that $r_{k,n}/\ga_{k,n}=a_k/\sqrt{2}+o(1)$ as $n\to \infty$ by \eqref{eq:sup0}.  It follows from \eqref{g2} that
\[
\left(\frac{(2+a_k)\delta_k}{2}-p\left(\delta_k-\delta\right)+o(1)\right)g(\mu_n)=\left(\frac{2+a_k}{\delta_k^{p-1}}+o(1)\right)\log{\frac1{\sqrt{\la_n}r_n}} 
\]
as $n\to \infty$. Since $g(u_n(r_n))/g(\mu_n)=\delta^p+o(1)$ as $n\to \infty$ from Lemma \ref{lem:g3}, we get the former equality. For the latter equality, it suffices to note that $g(u_n(r_n))^{-1}\log{\la_n}\to 0$ as $n\to \infty$ by \eqref{eq:sup2} and Lemma \ref{lem:g3}. We finish the proof. 
\end{proof}
Next, we consider the case (iii).
\begin{lemma}\label{lem:o6} Suppose $q=1$ and $k\in \mathbb{N}$. Select any sequence $(r_n)\subset [\rho_{k,n},\bar{\rho}_{k+1,n}]$ and  number $\eta\in[\eta_{k+1},\eta_k]$ such that
\[
u_n(r_n)=\mu_n-\left(\log{\frac1{\eta}}+o(1)\right)\frac{g(\mu_n)}{g'(\mu_n)}
\]
as $n\to \infty$. Then we have that 
\[
g(u_n(r_n))=(\alpha_{1,k}(\eta)+o(1))\log{\frac1{\sqrt{\la_n}r_n}}
\]
as $n\to \infty$. Especially, additionally supposing (H2), we get that
\[
g(u_n(r_n))=(\alpha_{1,k}(\eta)+o(1))\log{\frac1{r_n}}
\]
as $n\to \infty$.
\end{lemma}
\begin{proof} Similarly to the previous proof, it follows from  \eqref{id2}, \eqref{eq:sup01} and \eqref{eq:sup02} with  Lemmas \ref{lem:g4} and \ref{lem:b3} that 
\[
\begin{split}
g'(\mu_n)&(u_n(\rho_{k,n})-u_n(r_n))\\
&=\left(\frac{2+a_k}{\eta_k}+o(1)\right)\left(\log{\frac{r_{k,n}}{\rho_{k,n}}}+\log{\frac{1}{\sqrt{\la_n}r_{k,n}}}-\log{\frac1{\sqrt{\la_n}r_n}}\right)\\
&=\left(\frac{2+a_k}{2}+o(1)\right)g(\mu_n)-\left(\frac{2+a_k}{\eta_k}+o(1)\right)\log{\frac1{\sqrt{\la_n}r_n}} 
\end{split}
\]
as $n\to \infty$ from Lemmas \ref{lem:o0} and \ref{lem:o01}. Hence recalling \eqref{eq:sup00}, we calculate
\[
\left(\frac{2+a_k}{2}-\log{\frac{\eta_k}{\eta}}+o(1)\right)g(\mu_n)=\left(\frac{2+a_k}{\eta_k}+o(1)\right)\log{\frac1{\sqrt{\la_n}r_n}} 
\]
as $n\to \infty$. Then using Lemma \ref{lem:g4}, we obtain the former formula. We readily show the latter one by noting $g(u_n(r_n))^{-1}\log{\la_n}\to 0$ as $n\to \infty$ from \eqref{eq:sup2} and Lemma \ref{lem:g4}. We finish the proof.
\end{proof}
We show Theorem \ref{th:osc}.
\begin{proof}[Proof of Theorem \ref{th:osc}] The theorem is the consequence of Lemmas \ref{lem:o4}, \ref{lem:o5}, and \ref{lem:o6}.  We complete the proof. 
\end{proof}
We prove the  corollary.
\begin{proof}[Proof of Corollary \ref{cor:tb}] Fix any $k\in \mathbb{N}$.  For the first assertion, we choose $(r_n)=(r_{k,n})$ in (i) of Theorem \ref{th:osc} and note the fact that $r_{k,n}/\ga_{k,n}=a_k/\sqrt{2}+o(1)$ by \eqref{eq:sup0} and thus, 
\[
\frac{2a_k^2b_k(r_{k,n}/\ga_{k,n})^{a_k}}{(1+b_k(r_{k,n}/\ga_{k,n})^{a_k})^2}=\frac{a_k^2}{2}+o(1)
\]
as $n\to \infty$. This shows the desired conclusion. The rest of the assertions readily follows from  (ii) and (iii) of Theorem \ref{th:osc}. We finish the proof.  
\end{proof}
\subsection{Proof of Proposition \ref{cor:o2}}
We next prove the intersection property Proposition \ref{cor:o2}.
\begin{proof}[Proof of Proposition \ref{cor:o2}] Noting (H2) and Lemma \ref{lem:kap}, we set $\bar{\la}=\lim\sup_{n\to\infty}\la_n<\infty$.  For any number $k\in \mathbb{N}$, choose the sequences $(r_{k,n})$ and $(r_{k,n}^*)$ as in  Corollary \ref{cor:tb}. Take any values $\alpha\in(\alpha_{q,k}^*,2)$ and $\beta<\log{\{a_k^2/(2\bar{\la}h(0))\}}$. Then from the corollary and the condition (C), we get
\[
g(U(r_{k,n}))+\log{g'(U(r_{k,n}))}\le U_{2,\beta}(r_{k,n})< g(u_n(r_{k,n}))+\log{g'(u_n(r_{k,n}))}
\]
and 
\[
g(u_n(r_{k,n}^*))< U_{\alpha,0}(r_{k,n}^*)\le g(U(r_{k,n}^*))
\]
for all large $n\in \mathbb{N}$. Moreover,  if $k\ge2$, we also have that
\[
g(u_n(r_{k-1,n}^*))< U_{\alpha,0}(r_{k-1,n}^*)\le g(U(r_{k-1,n}^*))
\]
for all large $n\in \mathbb{N}$ since $\alpha^*_{q,k-1}<\alpha^*_{q,k}$. Hence noting also the fact that $U(r)\to \infty$ as $r\to 0^+$ if  $k=1$, we get from the monotonicity of $g$ and $g'$ by Lemma \ref{lem:monog} that there exist sequences $(r_{k,n}^\pm)\subset (r_{k-1,n}^*,r_{k,n}^*)$, where we put  $r_{0,n}^*=0$ for all $n\in \mathbb{N}$ for the case $k=1$, such that $r_{k,n}^-<r_{k,n}<r_{k,n}^+$ and  $u_n(r_{k,n}^\pm)=U(r_{k,n}^\pm)$ for all large $n\in \mathbb{N}$. 

Then we claim that $u_n(r_{k,n}^\pm)/\mu_n\to \delta_k$ as $n\to \infty$ for all $q\ge1$. For the case $q=1$, this is clear since 
\[
\frac{u_n(r_{k-1,n})}{\mu_n}\ge \frac{u_n(r_{k,n}^\pm)}{\mu_n}\ge\frac{u_n(r_{k+1,n})}{\mu_n}
\]
for all large $n\in \mathbb{N}$ and both sides converge to 1 as $n\to \infty$ where we put $r_{0,n}=0$ for all $n\in \mathbb{N}$ for the case $k=1$.  For the case $q>1$, we argue by contradiction. If the claim does not hold,  then there exists a constant $\delta\in[\delta_k^*,\delta_{k-1}^*]\setminus\{\delta_k\}$, where we put $\delta_0^*=1$ for the case $k=1$, such that $u_n(r_{k,n}^\pm)/\mu_n\to \delta$ as $n\to \infty$ up to a subsequence. Here we clearly get the contradiction for the sequence $(r_{1,n}^-)$ since $r_{1,n}^-<r_{1,n}$ for all large $n\in \mathbb{N}$. For the other cases, the previous fact implies that  $(r_{k,n}^-)\subset (r_{k-1,n}^*,\bar{\rho}_{k,n})$ and $(r_{k,n}^+)\subset (\rho_{k,n},r_{k,n}^*)$ for all  $n\in \mathbb{N}$ up to a subsequence. Then from (ii) of Theorem \ref{th:osc} and the condition (C), there exists a value $\alpha'\in(0,2)$ such that
\[
g(u_n(r_{k,n}^\pm))<U_{\alpha',0}(u_n(r_{k,n}^\pm))\le g(U(r_{k,n}^\pm)) 
\]
for all $n\in \mathbb{N}$ up to a subsequence. This with  the monotonicity of $g$ yields that $u_n(r_{k,n}^\pm)\not=U(r_{k,n}^\pm)$ for all  $n\in \mathbb{N}$. This is a contradiction. Furthermore,  assuming $q=1$, we can similarly  confirm that
\[
u_n(r_{k,n}^\pm)=\mu_n-\left(\log{\frac1{\eta_k}}+o(1)\right)\frac{g(\mu_n)}{g'(\mu_n)}
\]  
as $n\to \infty$. Indeed, if the conclusion fails, there exists a value $\eta\in[\eta_k^*,\eta_{k-1}^*]\setminus\{\eta_k\}$, where we set $\eta_0^*=1$ for the case $k=1$, such that 
\[
u_n(r_{k,n}^\pm)=\mu_n-\left(\log{\frac1{\eta}}+o(1)\right)\frac{g(\mu_n)}{g'(\mu_n)}
\]  
as $n\to \infty$ up to a subsequence. Again we clearly get a contradiction for the sequence $(r_{1,n}^-)$ from the fact that $r_{1,n}^-<r_{1,n}$ for all large $n\in \mathbb{N}$ with \eqref{eq:sup00}. For the other cases, we notice that $(r_{k,n}^-)\subset (r_{k-1,n}^*,\bar{\rho}_{k,n})$ and $(r_{k,n}^+)\subset (\rho_{k,n},r_{k,n}^*)$ for all  $n\in \mathbb{N}$ up to a subsequence by \eqref{eq:sup00} again. It follows from (iii) of Theorem \ref{th:osc} and the condition (C) that there exists a value $\alpha''\in (0,2)$ such that  
\[
g(u_n(r_{k,n}^\pm))<U_{\alpha'',0}(u_n(r_{k,n}^\pm))\le g(U(r_{k,n}^\pm)) 
\]
for all  $n\in \mathbb{N}$ up to a subsequence. This again leads to the similar contradiction as above. 

Finally, for the final conclusion, we take any sequence $(r_n)$ satisfying the assumption. If the assertion fails, there exists a number $k_0\in \mathbb{N}$ such that 
\[
Z_{(0,r_n)}[u_n-U]\le 2k_0
\]
for all $n\in \mathbb{N}$ up to a subsequence.  Fix any $k\in \mathbb{N}$ with $k\ge k_0+1$. Then, it follows from the previous conclusions that there exist $2k$ sequences $0<r_{1,n}^-<r_{1,n}^+<r_{2,n}^-<r_{2,n}^+<\cdots<r_{k,n}^-<r_{k,n}^+<1$ such that $u_n(r_{i,n}^\pm)=U(r_{i,n}^\pm)$ for  all $n\in \mathbb{N}$ and all $i=1,\cdots,k$ up to a subsequence. Moreover, from the construction of the sequences, we may assume  $u_n(r_{k,n}^\pm)>u_n(r_n)$, which implies $r_{k,n}^\pm<r_n$, for all $n\in \mathbb{N}$. In particular, we obtain that 
\[
Z_{(0,r_n)}[u_n-U]\ge 2k
\]
for all  $n\in \mathbb{N}$ which is a contradiction.  We complete the proof.  
\end{proof}
Lastly, for the analysis of the oscillation behaviors around singular solutions in the next section, we apply the analogous result for the corresponding shooting type problems. We focus on the situation (H0) with $h\equiv1$. Let $(r_n,\mu_n,v_n)\in (0,\infty)\times (0,\infty)\times C^2([0,r_n])$ verify 
\begin{equation}\label{p:vn}
\begin{cases}
-v_n''-\frac 1r v_n'= f(v_n),\ v_n>0\text{ in }(0,r_n),\\
v_n(0)=\mu_n,\ v_n'(0)=0=v_n(r_n),
\end{cases}
\end{equation}
for all $n\in \mathbb{N}$. We call $\{(r_n,\mu_n,v_n)\}$ a sequence of solutions of \eqref{p:vn}. We give  the following. 
\begin{proposition}\label{cor:int2} Suppose (H1). Let $\{(r_n,\mu_n,v_n)\}$ be any sequence of solutions of \eqref{p:vn} such that  $\mu_n\to \infty$ as $n\to \infty$ and $V(r)$  a continuous function defined for all small $r>0$ verifying the condition (C). 
 Then, for any $k\in \mathbb{N}$, there exist sequences $(s_{k,n}^\pm)\subset (0,1)$ such that $s_{k,n}^-<s_{k,n}^+$ and $v_n(s_{k,n}^\pm)=V(s_{k,n}^\pm)$ for all large $n\in \mathbb{N}$ and $v_n(s_{k,n}^\pm)/\mu_n\to \delta_k$ and if $q=1$, 
\[
v_n(s_{k,n}^{\pm})=\mu_n-\left(\log{\frac1{\eta_k}}+o(1)\right)\frac{g(\mu_n)}{g'(\mu_n)}
\]
as $n\to \infty$ up to a subsequence. In particular, for any sequence $(s_n)\subset (0,r_n)$, in the interval where $V$ is defined, such that  $v_n(s_n)/\mu_n\to0$ if $q>1$ and $(\mu_n-v_n(s_n))g'(\mu_n)/g(\mu_n)\to \infty $ if $q=1$ as $n\to \infty$, we get 
\[
\lim_{n\to \infty}Z_{(0,s_n)}[v_n-V]=\infty.
\]
\end{proposition}

\begin{proof} Set $\la_n=r_n^2$, $u_n(r)=v_n(\sqrt{\la_n}r)$ for all $r\in [0,1]$ and all $n\in \mathbb{N}$. Then $\{(\la_n,\mu_n,u_n)\}$ is a sequence of solutions of \eqref{pn} such that $\mu_n\to \infty$ as $n\to \infty$. Hence from Corollary \ref{cor:tb}, by extracting a subsequence if necessary, for all $k\in \mathbb{N}$, we take sequences $(r_{k,n}),(r_{k,n}^*)\subset(0,1)$ of values satisfying all the assertions of the corollary. Putting $s_{k,n}=\sqrt{\la_n}r_{k,n}$ and $s_{k,n}^*=\sqrt{\la_n}r_{k,n}^*$ for all $n\in \mathbb{N}$, we have that
\[
\begin{split}
g&(v_n(s_{k,n}))+\log{g'\left(v_n(s_{k,n})\right)}=2\log{\frac1{s_{k,n}}}+\log{\frac{a_k^2 }{2 h(0)}}+o(1)
\end{split}
\]
with
\[
\begin{cases}
\displaystyle\frac{v_n(s_{k,n})}{\mu_n}\to \delta_k &\text{ if $q>1$},\\
\displaystyle v_n(s_{k,n})=\mu_n-\left(\log{\frac1{\eta_k}}+o(1)\right)\frac{g(\mu_n)}{g'(\mu_n)}&\text{ if $q=1$,}  
\end{cases}
\]
and
\[
g(u_n(s_{k,n}^*))=\left(\alpha_{q,k}^{*}+o(1)\right)\log{\frac1{s_{k,n}^*}}
\] 
with
\[
\begin{cases}
\displaystyle\frac{v_n(s_{k,n}^*)}{\mu_n}\to \delta_k^* &\text{ if $q>1$},\\
\displaystyle v_n(s_{k,n}^*)=\mu_n-\left(\log{\frac1{\eta_k^*}}+o(1)\right)\frac{g(\mu_n)}{g'(\mu_n)}&\text{ if $q=1$,}  
\end{cases}
\]
as $n\to \infty$. Notice that Lemma \ref{lem:o01} implies $s_{k,n}<s_{k,n}^*<s_{k+1,n}\to0$ as $n\to \infty$. Then arguing as in the previous proof, we find a sequence $(s_{k,n}^\pm)\subset (s_{k-1,n}^*,s_{k,n}^*)$, where we put $s_{0,n}^*=0$ for all $n\in \mathbb{N}$ if $k=1$, such that $s_{k,n}^-<s_{k,n}<s_{k,n}^+$ and  $v_n(s_{k,n}^\pm)=V(s_{k,n}^\pm)$ for all large $n\in \mathbb{N}$ and $v_n(s_{k,n}^\pm)/\mu_n\to \delta_k$ and  if $q=1$,
\[
v_n(s_{k,n}^\pm)=\mu_n-\left(\log{\frac1{\eta_k}}+o(1)\right)\frac{g(\mu_n)}{g'(\mu_n)}
\]
as $n\to \infty$. The rest of the proof is same with that for the previous proposition. We finish the proof. 
\end{proof}
\section{Infinite oscillations around singular solutions}\label{sec:bif}
In this final section,  we prove all the results in Subsection  \ref{sub:bd}. To this end, we apply the idea of the proof of  Lemma 3.5  in \cite{Mi} with our intersection assertions in the previous sections. From now on, we focus on the case  $h\equiv1$  and consider solutions $(r_0,v)\in (0,\infty)\times  C^2([0,r_0])$ to the problem
\begin{equation}\label{pv}
\begin{cases}
-v''-\frac1r v'=f(v),\ v>0\text{ in }(0,r_0),\\
v'(0)=0=v(r_0),
\end{cases}
\end{equation}
and solutions $(R_0,V)\in (0,\infty)\times  C^2((0,R_0])$ to   
\begin{equation}\label{eq:sgv}
\begin{cases}
-V''-\frac1r V'=f(V),\ V>0\text{ in }(0,R_0),\\
\lim_{r\to 0^+}V(r)=\infty,\ V(R_0)=0.
\end{cases}
\end{equation}
We note that if $(\la,u)$ and $(\la,U)$ are solutions to \eqref{p} and \eqref{eq:sg}, then $(r_0,v)=(\sqrt{\la},u(\cdot/\sqrt{\la}))$ and $(R_0,V)=(\sqrt{\la},U(\cdot/\sqrt{\la}))$ are solutions to \eqref{pv} and \eqref{eq:sgv} respectively. 
\subsection{Proof of Theorem \ref{th:bif}}
We begin with the proof of Theorem \ref{th:bif} and its corollary. We first show that the assumption of the theorem implies the divergence of the intersection number between the corresponding blow-up solutions of \eqref{pv} and solutions of \eqref{eq:sgv}.  
\begin{lemma}\label{lem:bif1} Assume as in Theorem \ref{th:bif} and set $(r(\mu),v(\mu,\cdot))=(\sqrt{\la(\mu)},u(\mu,\cdot/\sqrt{\la(\mu)}))$ for all $\mu>0$ and $(r^*,V^*)=(\sqrt{\la^*},U^*(\cdot/\sqrt{\la^*}))$. Then we have that 
\[
\lim_{\mu\to \infty}Z_{(0,\min\{r(\mu),r^*\})}[v(\mu,\cdot)-V^*]=\infty.
\]
\end{lemma}
\begin{proof} We argue by contradiction. If the conclusion fails, then there exists a sequence $(\mu_n)$ of positive values such that $\mu_n\to \infty$ as $n\to \infty$ and 
\begin{equation}\label{eq:bi0}
\sup_{n\in \mathbb{N}}Z_{(0,\min\{r(\mu_n),r^*\})}[v(\mu_n,\cdot)-V^*]<\infty.
\end{equation}
We put $(\la_n,u_n)=(\la(\mu_n),u(\mu_n,\cdot))$ and $(r_n,v_n)=(r(\mu_n),v_n(\mu_n,\cdot))$ for all $n\in \mathbb{N}$. Then $\{(\la_n,\mu_n,u_n)\}$ and $\{(r_n,\mu_n,v_n)\}$ are sequences of solutions of \eqref{pn} and \eqref{p:vn} respectively. Put $s_n=\min\{r_n,r^*\}$ for all $n\in \mathbb{N}$. If $r_n\le r^*$ for all $n\in \mathbb{N}$, then $v_n(s_n)=0$ for all $n\in \mathbb{N}$. In particular, all the conditions in Proposition \ref{cor:int2} are satisfied with the above sequences $\{(r_n,\mu_n,v_n)\}$ and $(s_n)$ and the function $V=V^*$ where we noted \eqref{g2} for the case $q=1$. Hence the final conclusion of the proposition leads to the contradiction with  \eqref{eq:bi0}. Therefore, we may assume $s_n=r^*$ for all $n\in \mathbb{N}$ by extracting a subsequence if necessary. Then, for any $k\in \mathbb{N},$ by choosing a subsequence if necessary, we take the sequence $(r_{k,n})$ as in Theorem \ref{th:sup1}. We get from  \eqref{eq:sup2} and \eqref{eq:sup3} that
\[
\lim_{n\to \infty}\frac{\log{r_{k,n}}}{g(\mu_n)}=-\frac{\eta_k}{2}<0=\lim_{n\to \infty}\frac{\log{\frac{s_n}{\sqrt{\la_n}}}}{g(\mu_n)}.
\]
This implies that $r_{k,n}<s_n/\sqrt{\la_n}$ for all large $n\in \mathbb{N}$ and thus, 
\[
 \limsup_{n\to \infty}\frac{v_n(s_n)}{\mu_n}\le\lim_{n\to \infty}\frac{u_n(r_{k,n})}{\mu_n}= \delta_k
\]
if $q>1$ and 
\[
\liminf_{n\to \infty}(\mu_n-v_n(s_n))\frac{g'(\mu_n)}{g(\mu_n)}\ge\lim_{n\to \infty} (\mu_n-u_n(r_{k,n})))\frac{g'(\mu_n)}{g(\mu_n)}=\log{\frac1{\eta_k}}
\]
if $q=1$. Since $k\in \mathbb{N}$ is arbitrary, from Lemma \ref{lem:b2}, we have that 
\[
\lim_{n\to \infty}\frac{v_n(s_n)}{\mu_n}=0
\]
if $q>1$ and
\[
\lim_{n\to \infty}(\mu_n-v_n(s_n))\frac{g'(\mu_n)}{g(\mu_n)}=\infty
\]
if $q=1$. Consequently, we again confirm that all the assumptions in Proposition \ref{cor:int2} are verified and get the contradiction with \eqref{eq:bi0}.  We finish the proof. 
\end{proof}
We next prove that the divergence of the intersection number leads to  infinite  boundary oscillations. The original idea comes from the proof of Lemma 3.5 in \cite{Mi}. See also that of Lemma 5 in \cite{Mi2}. We remark that the second assertion (ii)' gives additional information of an infinite oscillation of the first derivative at the boundary.  
\begin{lemma}\label{lem:bif2} Assume $f\in C^2([0,\infty))$ with $f>0$ on $(0,\infty)$, and $h\equiv1$. Let  $(\la(\mu),u(\mu,\cdot))$ be the solutions $C^1$-curve obtained as in the first paragraph of Subsection \ref{sub:bd} and $(\la,U)=(\la^*,U^*)$  any solution of \eqref{eq:sg}. Put $(r(\mu),v(\mu,\cdot))=(\sqrt{\la(\mu)},u(\mu,\cdot/\sqrt{\la(\mu)}))$ for all $\mu>0$ and $(r^*,V^*)=(\sqrt{\la^*},U^*(\cdot/\sqrt{\la^*}))$. Finally, assume that
\begin{equation}\label{eq:Ninf}
\lim_{\mu\to \infty}Z_{(0,\min\{r(\mu),r^*\})}[v(\mu,\cdot)-V^*]=\infty.
\end{equation}
Then  the next assertions  (i)' and (ii)' hold true.
\begin{enumerate}
\item[(i)'] There exist sequences $(\mu_n^{\pm})$ of positive values such that $\mu_n^\pm\to \infty$ as $n\to \infty$ and $r(\mu_n^-)<r^*<r(\mu_n^+)$ for all $n\in \mathbb{N}$.
 \item[(ii)'] There exist sequences $(\nu_n^{\pm})$ of positive values such that $\nu_n^\pm\to \infty$ as $n\to \infty$  and  $r(\nu_n^\pm)=r^*$ and $|v_r(\nu_n^-,r^*)|<|(V^*)'(r^*)|<|v_r(\nu_n^+,r^*)|$ for all $n\in \mathbb{N}$ where $v_r(\mu,r)$ denote the first derivative of $v(\mu,r)$ with respect to $r$.
\end{enumerate}
\end{lemma}
\begin{proof}  Put $s(\mu)=\min\{r(\mu),r^*\}$ for all $\mu>0$. We note that $v(\mu,\cdot)$ and $V^*$ satisfy the same equation
\begin{equation}\label{eq:vV}
-v''-\frac1r v'=f(v)\text{ in }(0,s(\mu)]
\end{equation}
for all $\mu>0$. We first claim that for each $\mu>0$, 
\begin{equation}\label{eq:Nfin}
N(\mu):=Z_{(0,s(\mu)]}[v(\mu,\cdot)-V^*]<\infty.
\end{equation}
If not, for some $\mu>0$, there exists a sequence $(r_n)\subset (0,s(\mu)]$ of distinct values and a number $r_0\in (0,s(\mu)]$ such that $v(\mu,r_n)=V^*(r_n)$ for all $n\in \mathbb{N}$ and $r_n\to r_0$ as $n\to \infty$. But this yields $v(\mu,r_0)=V^*(r_0)$ and $v_r(\mu,r_0)=(V^*)'(r_0)$. This is impossible in view of the uniqueness property of the initial value problem of \eqref{eq:vV}.  

Next, noting the uniqueness property of \eqref{eq:vV} again, we apply the implicit function theorem at each intersection point to show that  each intersection point continuously moves (or stays) and is isolated locally in $\mu>0$. Especially, any intersection point in $(0,s(\mu))$ does not vanishes and that in $(0,s(\mu)]$ does not  split into two or more intersection points locally in $\mu>0$. Then, noting also $\lim_{r\to 0^+}V^*(r)=\infty$, we know that $N(\mu)$ changes only when  $r(\mu)$ arrives at or leaves $r^*$ and at the moment, the change  occurs only by one. 

We note that  the observation above implies that there exist sequences $(\mu_n'),(\mu_n'')$ of positive values such that $\mu_n',\mu_n''\to \infty$ as $n\to \infty$ and $r(\mu_n')=r^*$ and $r(\mu_n'')\not=r^*$ for all $n\in \mathbb{N}$. In fact, if such a sequence $(\mu_n')$ does not exist, then there exists a value $\bar{\mu}>0$ such that $r(\mu)\not=r^*$ for all $\mu\ge \bar{\mu}$. From the continuity of $r(\mu)$, we have  $r(\mu)<r^*$ for all $\mu\ge \bar{\mu}$ or $r(\mu)>r^*$ for all $\mu\ge \bar{\mu}$. But then the above observation yields that $N(\mu)=N(\bar{\mu})<\infty$ for all $\mu\ge \bar{\mu}$ by \eqref{eq:Nfin}. This contradicts  \eqref{eq:Ninf}.  The existence of $(\mu_n'')$ can be shown similarly. 

Based on the observation above, we shall prove (i)' and (ii)', that is,  there exist sequences $(\mu_n^{\pm}),(\nu_n^{\pm})$ of positive values such that  $\mu_n^\pm,\nu_n^\pm\to \infty$ as $n\to \infty$ and  $r(\mu_n^-)<r^*<r(\mu_n^+)$, $r(\nu_n^\pm)=r^*$, $|v_r(\nu_n^-,r^*)|<|(V^*)'(r^*)|<|v_r(\nu_n^+,r^*)|$ for  all $n\in \mathbb{N}$. To this end, we consider the next $4$ cases. Set any number $\mu_0>0$. 
\vspace{0.1cm}\\\textbf{Case 1.} Assume $r(\mu_0)<r^*$. Then in view of the observation above, there exists a value $\mu_1>\mu_0$ such that $r(\mu)<r^*$ for all $\mu\in [\mu_0,\mu_1)$ and $r(\mu_1)=r^*$. Moreover, from the uniqueness property as above, we have $|v_r(\mu_1,r^*)|<|(V^*)'(r^*)|$ and $N(\mu_1)=N(\mu_0)+1$ or $|v_r(\mu_1,r^*)|>|(V^*)'(r^*)|$ and $N(\mu_1)=N(\mu_0)$.  
\vspace{0.1cm}\\ \textbf{Case 2.} Suppose $r(\mu_0)>r^*$. Then similarly, there exists a value $\mu_1>0$ such that $r(\mu)>r^*$ for all $\mu\in [\mu_0,\mu_1)$ and $r(\mu_1)=r^*$. Furthermore, we get $|v_r(\mu_1,r^*)|<|(V^*)'(r^*)|$ and $N(\mu_1)=N(\mu_0)$ or  $|v_r(\mu_1,r^*)|>|(V^*)'(r^*)|$ and $N(\mu_1)=N(\mu_0)+1$.
\vspace{0.1cm}\\ \textbf{Case 3.} Assume $r(\mu_0)=r^*$ and $|v_r(\mu_0,r^*)|<|(V^*)'(r^*)|$. Then,  noting the uniqueness property again, there exists a value $\mu_1\ge \mu_0$ such that $r(\mu)=r^*$, $|v_r(\mu,r^*)|<|(V^*)'(r^*)|$, and $N(\mu)=N(\mu_0)$ for all $r\in[\mu_0,\mu_1]$ and $r(\mu)$ leaves $r^*$ when $\mu$ increases from $\mu_1$. Especially, there exists a value $\mu_2>\mu_1$ such that $r(\mu_2)<r^*$ and $N(\mu_2)=N(\mu_0)-1$ or  $r(\mu_2)>r^*$ and $N(\mu_2)=N(\mu_0)$.\vspace{0.1cm}
\\ \textbf{Case 4.} Suppose $r(\mu_0)=r^*$ and $|v_r(\mu_0,r^*)|>|(V^*)'(r^*)|$. Then, analogously, there exists a value $\mu_1\ge \mu_0$ such that $r(\mu)=r^*$, $|v_r(\mu,r^*)|>|(V^*)'(r^*)|$, and $N(\mu)=N(\mu_0)$ for all $r\in[\mu_0,\mu_1]$ and $r(\mu)$ leaves $r^*$ when $\mu$ increases from $\mu_1$. In particular, there exists a value $\mu_2>\mu_1$ such that $r(\mu_2)<r^*$ and $N(\mu_2)=N(\mu_0)$ or  $r(\mu_2)>r^*$ and $N(\mu_2)=N(\mu_0)-1$.\vspace{0.1cm}

Now, assume the desired sequence $(\mu_n^+)$ as above does not exist. Then there exists a value $\mu_0>0$ such that $r(\mu)\le r^*$ for all $\mu\ge \mu_0$. From the observation above, we may assume $r(\mu_0)<r^*$. Then, $r(\mu)$ repeats the behaviors in Cases 1, 3, and 4 after $\mu$ increases from $\mu_0$. This yields that  $N(\mu)\le N(\mu_0)+1<\infty$ for all $\mu\ge \mu_0$ by \eqref{eq:Nfin}. This contradicts  \eqref{eq:Ninf}. Similarly, we can show the existence of the desired sequence $(\mu_n^-)$ by noting behaviors in Cases 2, 3, and 4. This completes (i)'.

Finally, suppose that the desired sequence $(\nu_n^+)$ above does not exist. Then there exists a value $\mu_0>0$ such that $\mu\ge \mu_0$ and $r(\mu)=r^*$ imply $|v_r(\mu,r^*)|<|(V^*)'(r^*)|$. From the previous assertion, we may suppose $r(\mu_0)<r^*$. Then since $r(\mu)$ repeats the behaviors in Cases 1,  2, and 3 after $\mu$ increases from $\mu_0$, we have that $N(\mu)\le N(\mu_0)+1$ for all $\mu\ge \mu_0$. This gives the similar contradiction as above. Similarly we prove the existence of the desired sequence $(\nu_n^+)$ by noting the behaviors in Cases 1, 2, and 4. This finishes (ii)'. We complete the proof. 
\end{proof}
\begin{remark}\label{rmk:bifa} In the proof of the existence of the sequences $(\mu_n^\pm)$ and $(\nu_n^\pm)$ above, we extended the interval of $\mu$ where $N(\mu)\le N(\mu_0)+1$ by noting the behaviors in Cases 1,2,3, and 4. But those local behaviors  in $\mu$ do not rigorously prove that the extension can be accomplished globally  in $\mu$ up to infinity. We did not give the corresponding discussion in order to emphasize the essential idea and avoid the longer and technical arguments.  For the reader's convenience, we  show a rigorous argument in Appendix \ref{sec:ap}.   
\end{remark}
These lemmas lead to a global picture of  our infinite concentration and oscillation phenomena on  blow-up solutions along the solutions curve. To see this, we first notice that an infinite number of  bumps are produced by the combination of the oscillation (i)' of the first zero point $r(\mu)$ and that (ii)'  of the first derivative $v_r(\mu,\cdot)$ at the zero point as the proof shows. For instance, recalling the behaviors in Cases 1,2,3, and 4 in the proof, we observe that beginning from Cases 1 to 3, then proceeding from 3 to 2,   then from 2 to 4, and finally coming back from 4 to 1, a pair $r^\pm$ of intersection points such that $0<r^-<r^+< \min\{r(\mu),r^*\}$,  $v(\mu,r^\pm)=V^*(r^\pm)$,  and $v(\mu,r)>V^*(r)$ for all $r\in(r^-,r^+)$ appears at some value $\mu>0$. Now, we assume that the oscillation has just produced the $k$-th bump at some $\mu=\mu_0$ where $k\in \mathbb{N}$ is some natural number. Then for all $\mu>\mu_0$, let $0<s_1^-(\mu)<s_1^+(\mu)<s_2^-(\mu)<s_2^+(\mu)<\cdots<s_k^-(\mu)<s_k^+(\mu)\le \min\{r(\mu),r^*\}$ be the first $2k$ intersection points on $(0,\min\{r(\mu),r^*\}]$ such that,  for every $i=1,\cdots,k$, $v(\mu,s_i^\pm(\mu))=V^*(s_i^\pm(\mu))$ and $v(\mu,r)>V^*(r)$ for all $r\in(s_i^-(\mu),s_i^+(\mu))$. In view of Lemma \ref{lem:bif1} and the implicit function theorem, without loss of the generality, we may suppose those points always exist for all  $\mu>\mu_0$ and are continuous functions of $\mu$. Then, we observe that the $k$-th bump climbs up the graphs of blow-up solutions and enters the inside of the first $k$ concentration regions when $\mu>0$ goes to infinity, that is, we have that 
\[
\liminf_{\mu\to \infty}\frac{v(\mu,s^+_k(\mu))}{\mu}\ge \delta_k
\] 
if $q\in(1,2)$ and 
\[
\limsup_{\mu\to \infty}(\mu-v(\mu,s^+_k(\mu)))\frac{g'(\mu)}{g(\mu)}\le \log{\frac1{\eta_k}}
\]
if $q=1$. Indeed, if this assertion does not hold, then there exist a value $\e>0$ and a sequence $(\mu_n)$ of positive values such that $\mu_n\to \infty$ as $n\to \infty$ and 
\[
\frac{v(\mu_n,s^+_k(\mu_n))}{\mu_n}\le \delta_k-\e 
\]
if $q\in(1,2)$ and 
\[
v(\mu_n,s_k^+(\mu_n))\le \mu_n -\left(\log{\frac1{\eta_k}}+\e \right)\frac{g(\mu_n)}{g'(\mu_n)} 
\]
if $q=1$ for all $n\in \mathbb{N}$. This implies that for any sequence $(s_n)$ such that $v(\mu_n,s_n)/\mu_n\to \delta_k$ if $q\in(1,2)$ and 
\[
v(\mu_n,s_n)=\mu_n-\left(\log{\frac1{\eta_k}}+o(1)\right)\frac{g(\mu_n)}{g'(\mu_n)}
\]
if $q=1$ as $n\to \infty$, we have that $s_n<s^+_k(\mu_n)$ for all large $n\in \mathbb{N}$ and thus, there are at most $2k-1$ intersection points in $(0,s_n]$ for all large $n\in \mathbb{N}$. On the other hand, the former assertion in Proposition \ref{cor:int2} shows that with the suitable choice of such a sequence $(s_n)$, there must be  at least $2k$ intersection points in $(0,s_n]$ up to a subsequence. This is a contradiction.  

The above observation gives that the boundary oscillations (i)' and (ii)' generate an infinite number of bumps one after another from  the boundary and each of them climbs up the graphs of blow-up solutions and finally enters the inside of the infinite concentration regions near the origin. Then, we can understand  that this infinite climbing-up behavior of the infinite sequence of bumps from the bottom to the top supplies the infinite sequence of bubbling parts near the origin detected via the suitable scaling procedure as in Theorem \ref{th:sup1}. 

Now we show the theorem. 
\begin{proof}[Proof of Theorem \ref{th:bif}] Put $(r(\mu),v(\mu,\cdot))=(\sqrt{\la(\mu)},u(\mu,\cdot/\sqrt{\la(\mu)}))$ for all $\mu>0$ and $(r^*,V^*)=(\sqrt{\la^*},U^*(\cdot/\sqrt{\la^*}))$. Then, from Lemma \ref{lem:bif1}, we have that 
\[
\lim_{\mu\to \infty}Z_{(0,\min\{r(\mu),r^*\})}[v(\mu,\cdot)-V^*]=\infty.
\] 
Consequently, from (i)' of Lemma \ref{lem:bif2} and the definitions of $r(\mu)$ and $r^*$, we get  (i). Moreover, since $|u_r(\nu_n^\pm,1)|=r^*|v_r(\nu_n^\pm,r^*)|$ and $|(U^*)'(1)|=r^*|(V^*)'(r^*)|$ for all $n\in \mathbb{N}$, we prove (ii) by (ii)'.  Finally, we show (iii) by contradiction. Assume that (iii) does not hold. Then there exists a sequence $(\mu_n)$ of  positive values such that $\mu_n\to \infty$ as $n\to \infty$ and 
\[
\sup_{n\in \mathbb{N}}Z_{(0,1)}[u(\mu_n,\cdot)-U^*]<\infty.
\]
But then,  since $\{(\la_n,\mu_n,u_n)\}=\{(\la(\mu_n),\mu_n,u(\mu_n,\cdot))\}$ is a sequence of solution of \eqref{pn}, we get the contradiction by Proposition \ref{cor:o2}. We finish the proof.  
\end{proof}
\begin{proof}[Proof of Theorem \ref{cor:bif}] From (i) of Theorem \ref{th:bif} and the continuity of $\la(\mu)$ there, we readily get the desired conclusion. We complete  the proof. 
\end{proof}
\subsection{Proof of Theorem \ref{th:bif2}}
We finally give the proof of Theorem \ref{th:bif2}. The next lemma is  a preliminary for the construction of solutions of \eqref{eq:sgv}.
\begin{lemma}\label{lem:cnstsg}
Assume that there exists a value $\bar{R}>0$ and a function $\bar{V}\in C^2((0,\bar{R}])$ such that
\[
-\bar{V}''-\frac1r \bar{V}'=f(V), \ \bar{V}>0\text{ in }(0,\bar{R}]
\]
and $\lim_{r\to 0^+}\bar{V}(r)=\infty$. Then there exists a solution $(R_0,V)=(R^*,V^*)$ of \eqref{eq:sgv} satisfying $V^*=\bar{V}$ on $(0,\bar{R}]$.
\end{lemma}
\begin{proof} Without loosing the generality, we may assume $\bar{V}'(\bar{R})<0$.  We consider the next initial value problem
\begin{equation}\label{eq:vbar}
\begin{cases}
-V''-\frac1r V'=f(V)\text{ in } [\bar{R},\infty),\\
V(\bar{R})=\bar{V}(\bar{R}), \ V'(\bar{R})=\bar{V}'(\bar{R}).
\end{cases}
\end{equation}
Thanks to (H0), there exists a value $\e>0$ such that \eqref{eq:vbar} admits a solution $V$ on $[\bar{R},\bar{R}+\e]$. We set
\[
R^*:=\sup\{r>\bar{R}\ |\ \text{\eqref{eq:vbar} admits a solution $V>0$ on $[\bar{R},r]$}\}.
\]
We first claim  that $V$ is strictly decreasing on $[\bar{R},R^*)$.  In fact, integrating the equation in \eqref{eq:vbar}, we get for all $r\in [\bar{R},R^*)$ that
\[
-rV'(r)=-\bar{R}V'(\bar{R})+\int_{\bar{R}}^{r}f(V(s))sds\ge -\bar{R}V'(\bar{R})>0
\]
from the nonnegativity of $f$ by (H0) and our assumption. This prove the claim. We next claim $R^*<\infty$. If not,  integrating the previous formula again, we have that for all $r>\bar{R}$, 
\[
0>-V(r)=-V(\bar{R})-RV'(R)\log{\frac{r}{\bar{R}}}\to \infty
\]
as $r\to \infty$. This is a contradiction.  Note that from the monotonicity, we have $\lim_{r\to (R^*)^-}V(r)=0$. Then, putting $V^*=\bar{V}$ on $(0,\bar{R})$ and $V^*=V$ on $[\bar{R},R^*)$ with $V^*(R^*)=0$, we complete the desired conclusion. We finish the proof.    
\end{proof}
We next show that (H3) implies the assumptions in \cite{FIRT} which give the existence of the desired singular solutions. To this end, we introduce their setting noted in Section 2 there. We define, for all $t\ge t_0$,
\[
F(t):=\int_t^\infty\frac1{f(s)}ds
\]
and then, 
\[
\frac1{B_1[f](t)}:=(-\log{F(t)})[1-f'(t)F(t)]
\]
and 
\[
\frac1{B_2[f](t)}:=\frac{\left(1-f'(t)F(t)\right)'}{\left(\frac1{-\log{F(t)}}\right)'}=f'(t)F(t)\left(-\log{F(t)}\right)^2\left[\frac{f(t)f''(t)F(t)}{f'(t)}-1\right].
\]
(We here use the same character $F$ for the different function from that in Subsections 2.1 and 2.2 in the former part \cite{N3} for the convenience.) Then in our notation, their first two assumptions can be written as follows.
\begin{enumerate}
\item[($f_1$)] $f\in C^2((t_0,\infty))$, $f(t)>0$, $F(t)<\infty$, and $f'(t)>0$ for all $t> t_0$. 
\item[($f_2$)] The limits $\lim_{t\to \infty}f'(t)F(t)$ and $B_2[f](t)$ exist and satisfy
\[
\lim_{t\to \infty}f'(t)F(t)=1
\] 
and
\[
B:=\lim_{t\to \infty}B_2[f](t)\in [1,\infty).
\]
\end{enumerate}
Moreover, setting a value $B'>1$ so that $1/B+1/B'=1$ if $B>1$, we put  functions
\[
f_0(t)=\begin{cases}\displaystyle\frac{4}{BB'}t^{1-2B'}e^{t^{B'}} &\text{ if }B>1,\vspace{0.2cm}\\
\displaystyle 4\frac{e^{e^{t}}}{e^{2t}} &\text{ if }B=1,
\end{cases}
\] 
and for each case of $B\ge1$,  
\[
F_0(t):=\int_t^\infty\frac1{f_0(s)}ds
\]
for all $t>0$. Moreover, we put $u_0$ as an explicit solution of the equation $-\Delta u=f_0(u)$ given by 
\[
u_0(r)=\begin{cases}\displaystyle\left(\log{\frac1{r^2}}\right)^{\frac1{B'}}&\text{ if }B>1,\vspace{0.2cm}\\
\displaystyle \log{\left(\log{\frac1{r^2}}\right)} &\text{ if }B=1,
\end{cases}
\]
for all $r\in(0,1)$.  As noted in (2.5) there, we  get for any $B\ge1$,
\[
F_0(u_0(r))=\frac{B}{4}r^2\left(\log{\frac{1}{r^2}}+1\right)=:w(r),
\]
and then define
\[
\tilde{u}(r):=F^{-1}[F_0(u_0(r))]
\]
for all $r>0$ in an appropriate region near the origin where $F^{-1}$ is the inverse function of $F$. Finally setting
\[
R_i(r):=\left|\frac1{B_i[f](\tilde{u}(r))}-\frac1{B_i[f_0](u_0(r))}\right|
\]
for $i=1,2$, we write their third condition as
\begin{enumerate}
\item[($f_3$)] $\lim_{r\to 0^+}(-\log{r})^{1/2}[R_1(r)+R_2(r)]=0$. 
\end{enumerate} 
Let us show that (H3) implies ($f_i$) for all $i=1,2,3$. We first give a note on \eqref{g3}. 
\begin{lemma}\label{lem:hatq} Assume  (H3). Then we get that $\hat{q}=(2-q)^{-1}$.
\end{lemma}
\begin{proof} From \eqref{g1} and the de l'H\^{o}spital rule, we have that
\[
\frac1q=\lim_{t\to \infty}\frac{(g(t)g''(t))'}{(g'(t)^2)'}=\lim_{t\to \infty}\left(\frac12+\frac{g(t)g''(t)}{2g'(t)^2}\frac{g'(t)g'''(t)}{g''(t)^2}\right).
\]
This gives the desired conclusion by \eqref{g1} and \eqref{g3}. We complete the proof. 
\end{proof}
Then we prove  the next desired lemma. 
\begin{lemma}\label{lem:bif3} (H3) implies ($f_1$), ($f_2$), and ($f_3$). Moreover, there exists a solution $(R_0,V)=(R^*,V^*)$ of \eqref{eq:sgv} such that  
\[
g(V^*(r))+\log{g(V^*(r))}=\log{\frac1{w(r)}}+o(1),
\]
as $r\to 0^+$. 
\end{lemma}
\begin{proof} From \eqref{g2}, there exists a value $c_0>0$ such that $(\log{g(t)})'\ge c_0(\log{t})'$ for all large $t\ge t_0$. This implies that $g(t)\ge c_1t^{c_0}$ for all large $t\ge t_0$ and some $c_1>0$. This gives $F(t)<\infty$ for all $t\ge t_0$. Hence, noting  (H0) and (H1),  we get ($f_1$). Moreover, thanks to the final condition in (H3), we can apply Lemma 4.1 in \cite{FIRT}. Consequently, we get that
\begin{align}
&\log{F(t)}=-g(t)-\log{g'(t)}+O\left(\frac{g''(t)}{g'(t)^2}\right),\label{eq:lgF}\\
&f(t)F(t)=\frac1{g'(t)}(1+o(1)),\label{eq:fF}\\
&f'(t)F(t)=1-\frac{g''(t)}{g'(t)^2}+\frac{3g''(t)^2-g'(t)g'''(t)}{g'(t)^4}(1+o(1)),\notag\\
&\frac{f(t)f''(t)F(t)}{f'(t)}=1+\frac{2g''(t)^2-g'(t)g'''(t)}{g'(t)^4}\notag\notag\\
&\ \ \ \ \ \ \ \ \ \ \ \ \ \ \ \ \ \ \ \ \ \ \ \ +O\left(\left|\frac{g''(t)g'''(t)}{g'(t)^5}\right|+\left|\frac{g''''(t)}{g'(t)^4}\right|+\left|\frac{g''(t)^3}{g'(t)^6}\right|\right),\notag
\end{align}
as $t\to \infty$. 
 It follows that
\[
f'(t)F(t)=1-\frac{g''(t)}{g'(t)^2}\left(1-\frac{3g''(t)}{g'(t)^2}+\frac{g''(t)}{g'(t)^2}\frac{g'(t)g'''(t)}{g''(t)^2}(1+o(1))\right)\to 1
\]
as $t\to \infty$ by \eqref{g1} and \eqref{g3} with Lemma \ref{lem:hatq}. Similarly, using \eqref{g1},  \eqref{g3} with Lemma \ref{lem:hatq},  \eqref{g4}, and \eqref{eq:lgg}, we compute
\begin{align}
&\frac1{B_2[f](t)}\notag\\
&=\frac{2g(t)^2g''(t)^2}{g'(t)^4}\left(1+O\left(\frac{g''(t)}{g'(t)^2}\right)\right)\left(1+\frac{\log{g'(t)}}{g(t)}+O\left(\frac{g''(t)}{g(t)g'(t)^2}\right)\right)^2\notag\\
&\ \ \ \ \ \times\left(1-\frac{g'(t)g'''(t)}{2g''(t)^2}+O\left(\left|\frac{g'''(t)}{g'(t)g''(t)}\right|+\left|\frac{g''''(t)}{g''(t)^2}\right|+\left|\frac{g''(t)}{g'(t)^2}\right|\right)\right)\label{eq:b2ft}\\
&\to \frac1q\notag
\end{align}
as $t\to \infty$. This proves $(f_2)$. Remark that this also shows $B=q$ and $B'=p$ if $B>1$ by Lemma \ref{lem:pq}. 

Let us finally show ($f_3$). We first note that
\[
\begin{split}
\frac1{B_1[f](t)}&=\frac{g(t)g''(t)}{g'(t)^2}\left(1+\frac{\log{g'(t)}}{g(t)}+O\left(\frac{g''(t)}{g(t)g'(t)^2}\right)\right)\\
&\ \ \ \ \ \ \ \ \ \ \ \ \ \ \ \ \ \ \times \left(1-\frac{3g''(t)}{g'(t)^2}+\frac{g''(t)}{g'(t)^2}\frac{g'(t)g'''(t)}{g''(t)^2}(1+o(1))\right)\\
&= \frac 1q+o\left(\frac1{\sqrt{g(t)}}\right)
\end{split}
\]
by \eqref{g5} and \eqref{g3} with Lemma \ref{lem:hatq} and from \eqref{eq:b2ft} with \eqref{g5}, \eqref{g7}, and \eqref{g6}, 
\[
\frac1{B_2[f](t)}=\frac1q+o\left(\frac1{\sqrt{g(t)}}\right)
\]
as $t\to \infty$. Moreover, from (4.14) and the formula on p43 in \cite{FIRT}, we know that
\[
\frac{1}{B_i[f_0](t)}=
\displaystyle1-\frac1p+O\left(\xi_q(t)\right)
\]
as $t\to \infty$ for each $i=1,2$  where $\xi_q(t)=t^{-p}\log{t}$ if $q>1$ and  $\xi_q(t)=te^{-t}$ and $1/p=0$ if $q=1$.  In addition, since $\tilde{u}(r)=F^{-1}(w(r))$ from the definition, we have that $\tilde{u}(r)\to \infty$ as $r\to 0^+$ by the monotonicity of $F$ and the fact that $F(t)\to0$ as $t\to \infty$. Consequently, we get that 
\begin{equation}\label{eq:Ri}
\begin{split}
R_i(r)&=
\displaystyle o\left(\frac1{\sqrt{g(\tilde{u}(r))}}\right)+O\left(\xi_q\left(u_0(r)\right)\right)
\end{split}
\end{equation}
as $r\to 0^+$ for each $i=1,2$. Finally, we notice from the definition of $\tilde{u}$ and $w$ that  
\[
(2+o(1))\log{\frac1r}=\log{\frac 1{w(r)}}=
\displaystyle \log{\frac1{F(\tilde{u}(r))}}=(1+o(1))g(\tilde{u}(r))
\]
by \eqref{eq:lgF}, \eqref{eq:lgg}, and \eqref{g1} and 
\[
\xi_q\left(u_0(r)\right)=O\left(\frac{\log{\left(\log{\frac1r}\right)}}{\log{\frac1r}}\right)
\]
as $r\to 0^+$. Therefore, we deduce that 
\[
\lim_{r\to 0^+}\left(\log{\frac1r}\right)^{\frac12}[R_1(r)+R_2(r)]=0
\]
which confirms ($f_3$). 

As a result, from Theorem 2.1 in \cite{FIRT} with Lemma \ref{lem:cnstsg} above, there exists  a solution  $(R_0,V)=(R^*,V^*)$ of \eqref{eq:sgv} such that 
\[
V^*(r)=\tilde{u}(r)+o\left(\frac1{g'(\tilde{u}(r))}\right)
\]
as $r\to 0^+$ by \eqref{eq:fF} and \eqref{eq:Ri}. Here, noting  Lemma \ref{lem:g2}, we have that
\[
g(V^*(r))=g(\tilde{u}(r))+o\left(\frac{g'\left(\tilde{u}(r)+o\left(\frac1{g'(\tilde{u}(r))}\right)\right)}{g'(\tilde{u}(r))}\right)=g(\tilde{u}(r))+o(1)
\]
and, for some function $y(r)=\tilde{u}(r)+o(1/g'(\tilde{u}(r)))$, that
\[
\begin{split}
\log{g'(V^*(r))}&=\log{g'(\tilde{u}(r))}+o\left(\frac1{g'(\tilde{u}(r))}\frac{g''\left(y(r)\right)}{g'\left(y(r)\right)}\right)\\
&=\log{g'(\tilde{u}(r))}+o(1)
\end{split}
\]
as $r\to 0^+$ where we used also \eqref{g1} for the last equality. Therefore, by \eqref{eq:lgF}, we deduce
\[
\begin{split}
g(V^*(r))+\log{g'(V^*(r))}&=g(\tilde{u}(r))+\log{g'(\tilde{u}(r))}+o(1)
=\log{\frac1{w(r)}}+o(1)
\end{split}
\]
as $r\to 0^+$. This completes the proof.   
\end{proof}
Finally, we prove the  theorem and its corollary. 
\begin{proof}[Proof of Theorem \ref{th:bif2}]Take a solution $(R^*,V^*)$ of \eqref{eq:sgv} from Lemma \ref{lem:bif3}. Set $(\la^*,U^*)=((R^*)^2,V^*(R^*\cdot))\in(0,\infty)\times C^2((0,1])$. Then, $(\la,U)=(\la^*,U^*)$ is a solution of \eqref{eq:sg} such that $U^*$ satisfies the condition (C). As a consequence, Theorem \ref{th:bif}  gives the desired conclusion.  We finish the proof. 
\end{proof}
\begin{proof}[Proof of Corollary \ref{cor:bif3}]
The proof follows from the assertion (i) of the theorem with the continuity of $\la(\mu)$. We finish the proof. 
\end{proof}
\begin{remark}\label{rmk:41}
To see that the examples (a) and (b) under Corollary \ref{cor:bif3} satisfy all the assumptions in Lemma 4.1 in \cite{FIRT}, for any $C^5$ function $g(t)$ defined for all $t\ge t_0$, we set $g_1(t)=1/g'(t)$, $g_{i+1}(t)=g_i'(t)/g'(t)$ for all $i=1,2,3,4$ and all $t\ge t_0$. Then the direct computation shows that 
\[
\begin{split}
&g_1(t)=\frac1{g'(t)},\ g_2(t)=-\frac{g''(t)}{g'(t)^3},\ g_3(t)=\frac{-g'(t)g'''(t)+3g''(t)^2}{g'(t)^5},\\
&g_4(t)=\frac{-g'(t)^2g''''(t)+10g'(t)g''(t)g'''(t)-15g''(t)^3}{g'(t)^7},\\
&g_5(t)=\frac{-g'(t)^3g'''''(t)+10g'(t)^2g'''(t)^2+15g'(t)^2g''(t)g''''(t)}{g'(t)^9}\\
&\ \ \ \ \ \ \ \ \ \ +\frac{-105 g'(t)g''(t)^2g'''(t)+105g''(t)^4}{g'(t)^9}
\end{split}
\]
for all $t\ge t_0$. Then, for the examples (a) and (b), putting $g(t)=t^p+ct^{\bar{p}}+m\log{t}$ and $g(t)=e^t+w(t)$, we get $g_i(t)=c_it^{1-ip}(1+o(1))$ and $g_i(t)=c_i e^{-it}(1+o(1))$ as $t\to \infty$ for some constant $c_i\not=0$ and all $i=1,2,3,4,5$ respectively. Then one easily checks the desired conclusion.
\end{remark}
\appendix
\section{Note on the proof of Lemma \ref{lem:bif2}}\label{sec:ap}
In this appendix, we give a rigorous argument for the proof of the existence of the desired sequences in (i)' and (ii)' of Lemma \ref{lem:bif2} as noted in Remark \ref{rmk:bifa}. Noting the basic observation in the first three paragraphs and the behaviors in Cases 3 and 4 in the proof of the lemma, we argue as follows. 

Let us first show the existence of the desired sequence $(\mu_n^+)$. Assume it does not exist. Then there exists a value $\mu_0>0$ such that $r(\mu)\le r^*$ for all $\mu\ge \mu_0$. From the basic observation, we can assume $r(\mu_0)<r^*$. 
 Set $\mu_1=\sup\{\mu\ge \mu_0\ |\ N(m)\le N(\mu_0)+1\text{ for all }m\in[\mu_0,\mu]\}$. The continuity of $r(\mu)$ and \eqref{eq:Ninf} yield $\mu_0<\mu_1<\infty$. From the continuity again and the basic observation with the behaviors in Cases 3 and 4, we get $r(\mu_1)=r^*$ and $N(\mu_1)=N(\mu_0)+2$. First assume $|v_r(\mu_1,r^*)|>|(V^*)'(r^*)|$. Then from the assumption, as $\mu$ decreases from $\mu_1$, $r(\mu)$ stays at $r^*$ or moves to the left. 
In both cases, $N(\mu)$ does not change. This contradicts the definition of $\mu_1$. Hence, we have $|v_r(\mu_1,r^*)|<|(V^*)'(r^*)|$. Then when $\mu$ decreases from $\mu_1$, $r(\mu)$ must move to the left in order that $N(\mu)$ decreases by one. In particular, there exists a value $\mu_0<\mu_2<\mu_1$ such that $r(\mu_2)<r^*$ and $N(\mu_2)=N(\mu_0)+1$. Put $\mu_3:=\inf\{\mu\le \mu_2\ |\ N(m)=N(\mu_0)+1\text{ for all }m\in[\mu,\mu_2]\}$. The continuity of $r(\mu)$ yields $\mu_3\in(\mu_0,\mu_2)$. Then from the continuity again and the behaviors in Cases 3 and 4, we have that $r(\mu_3)=r^*$ and $N(\mu_3)=N(\mu_0)+1$.  Assume $|v_r(\mu_3,r^*)|<|(V^*)'(r^*)|$. Then since $r(\mu_2)<r^*$, 
 there exists a value $\mu_4\in [\mu_3,\mu_2)$ such that $r(\mu)=r^*$, $N(\mu)=N(\mu_0)+1$, and $|v_r(\mu,r^*)|<|(V^*)'(r^*)|$ for all $\mu\in[\mu_3,\mu_4]$ and when $\mu$ increases from $\mu_4$, $r(\mu)$ leaves $r^*$ to the left and at the moment, $N(\mu)$ decreases to $N(\mu_0)$. This is a contradiction since $\mu_4\in[\mu_3,\mu_2)$. Hence we have  $|v_r(\mu_3,r^*)|>|(V^*)'(r^*)|$. But then, when $\mu$ decreases from $\mu_3$, $r(\mu)$ stays at $r^*$ or moves to the left. In both cases, $N(\mu)$ does not change which contradicts the definition of $\mu_3$. Hence we prove the existence of the desired sequence $(\mu_n^+)$. The proof of the existence of the sequence $(\mu_n^-)$ can be similarly done. 

Next we show the existence of the sequence $(\nu_n^+)$. Assume the desired sequence $(\nu_n^+)$ does not exist. Then  there exists a value $\mu_0>0$ such that $\mu\ge \mu_0$ and $r(\mu)=r^*$ imply  $|v_r(\mu,r^*)|<|(V^*)'(r^*)|$. Noting the previous assertions, we can assume $r(\mu_0)<r^*$. Again put $\mu_1=\sup\{\mu\ge \mu_0\ |\ N(m)\le N(\mu_0)+1\text{ for all }m\in[\mu_0,\mu]\}$. The continuity of $r(\mu)$ and \eqref{eq:Ninf} ensure $\mu_1\in(\mu_0,\infty)$. By the continuity again and the assumption, we get $r(\mu_1)=r^*$ and  $|v_r(\mu_1,r^*)|<|(V^*)'(r^*)|$. Moreover, the behavior in Case 3 and the definition of $\mu_1$ yield $N(\mu_1)=N(\mu_0)+2$. Then, again from the definition of $\mu_1$,  when $\mu$ decreases from $\mu_1$, $r(\mu)$ must leave $r^*$ to the left in order that $N(\mu)$ decreases by one. In particular, there exists a value $\mu_2\in(\mu_0,\mu_1)$ such that $r(\mu_2)<r^*$ and $N(\mu_2)=N(\mu_0)+1$.   Put $\mu_3=\inf\{\mu \le \mu_2\ |\ N(m)=N(\mu_0)+1\text{ for all }m\in[\mu,\mu_2]\}$. From the continuity of $r(\mu)$, we have $\mu_3\in(\mu_0,\mu_2)$. It follows from the continuity again and the behavior in Case 3 that  $r(\mu_3)=r^*$, $|v_r(\mu_3,r^*)|<|(V^*)'(r^*)|$, and $N(\mu)=N(\mu_0)+1$. Then 
 since $r(\mu_2)<r^*$, there exists a value $\mu_4\in[\mu_3,\mu_2)$ such that $r(\mu)\ge r^*$ and $N(\mu)=N(\mu_0)+1$ for  all $\mu\in[\mu_3,\mu_4]$, $r(\mu_4)=r^*$, $|v_r(\mu_4,r^*)|<|(V^*)'(r^*)|$, and when $\mu$ increases from $\mu_4$, $r(\mu)$ leaves $r^*$ to the left which implies that $N(\mu)$ decrease by one. This is  a contradiction since   $\mu_3\le \mu_4<\mu_2$. This proves the existence of the sequence $(\nu_n^+)$. We can similarly prove the existence of  $(\nu_n^-)$. We complete the proof.  
\subsection*{Acknowledgement} The author is grateful to Elide Terraneo for her helpful suggestion concerning Theorem \ref{th:bif2}. This work is supported by JSPS KAKENHI Grant Numbers 21K13813. 

\end{document}